\documentclass[hidelinks,onefignum,onetabnum]{siamart220329}

\usepackage{lipsum}
\usepackage{amsfonts}
\usepackage{graphicx}
\usepackage{epstopdf}
\usepackage{algorithmic}
\ifpdf
  \DeclareGraphicsExtensions{.eps,.pdf,.png,.jpg}
\else
  \DeclareGraphicsExtensions{.eps}
\fi

\newsiamremark{remark}{Remark}
\newsiamremark{hypothesis}{Hypothesis}
\crefname{hypothesis}{Hypothesis}{Hypotheses}
\newsiamthm{claim}{Claim}

\headers{ Lagrange multiplier approach for structure preserving schemes}{Q. Cheng, T. Wang, and X. Zhao}

\title{Arbitrary high-order maximum principle-preserving and energy dissipating schemes for gradient flows \thanks{Submitted to the editors DATE.
\funding{Q. Cheng is supported by NSFC 12301522 and the Fundamental Research Funds for the Central Universities. T. Wang and X. Zhao are supported by National Key Research and Development Program of China
(No. 2024YFE03240400), National MCF Energy R\&D Program and NSFC 42450275, 12271413.}}}

\author{Qing Cheng\thanks{School of Mathematical Sciences, Tongji University, Shanghai, China 
  (\email{qingcheng@tongji.edu.cn}).} 
\and Tingfeng Wang\thanks{School of Mathematics and Statistics \& Computational Sciences Hubei Key Laboratory, Wuhan University, Wuhan, China 
  (\email{tingfengwang@whu.edu.cn}, \email{matzhxf@whu.edu.cn}).}
\and Xiaofei Zhao\footnotemark[3]}

\usepackage{amsopn}

\ifpdf
\hypersetup{
	pdftitle={A new Lagrange multiplier approach for constructing structure preserving schemes, III. Bound preserving and energy dissipating.},
	pdfauthor={Q. Cheng, T. Wang, and X. Zhao}
}
\fi

\usepackage{amssymb}
\usepackage{enumerate} 
\usepackage{enumitem} 
\usepackage{booktabs,multirow}  
\usepackage{comment}
\usepackage{tabularx}  
\usepackage{float} 
\usepackage{subfig}
\newtheorem{example}[theorem]{Example}
\newtheorem{assumption}{Assumption}
\newcommand{\fe}{\mathrm{e}}
\NewDocumentEnvironment{customproof}{O{Proof}}{%
    \begin{proof}[#1] 
    }{%
    \end{proof}
}

\allowdisplaybreaks[3]

\begin{document}
    \maketitle
    
    \begin{abstract}
For gradient flows, the existing structure-preserving schemes are difficult to achieve arbitrary high-order accuracy in time while preserving maximum-principle (MBP) and energy dissipating simultaneously. In this paper, we develop a new framework for constructing structure-preserving schemes which shall preserve those nice properties. By introducing  KKT-conditions for energy dissipating and bound-preserving, we rewrite the original gradient flow into an expanded and coupled system. We shall utilize a novel predictor-corrector-corrector framework, termed the PCC method, which consists of a prediction  from any numerical scheme to the user's favor, followed by two correction steps designed to enforce energy stability and MBP, respectively. We take the exponential time differencing Runge-Kutta scheme (ETDRK) as an example and establish the unique solvability and  robust error analysis for our new framework. Extensive numerical experiments are provided to validate the efficiency and accuracy of our new approach. Enough numerical comparisons with the existing popular schemes are shown that our structure-preserving schemes can avoid numerical oscillations and capture the exact evolution of energy.
    \end{abstract}
    
    \section{Introduction}
    Gradient flows, characterized by the evolution of a system toward a state of lower free energy, are ubiquitous in the mathematical modeling of physical, biological, and engineering processes. In this work, we consider a free energy functional $\mathcal{E}[\phi]$ defined on a bounded spatial domain $\Omega\subset\mathbb{R}^{d}$ (with $d=1,2,3$) of the form
    \begin{align}
        \mathcal{E}[\phi(\mathbf{x},t)] = \int_{\Omega} \left( \frac{1}{2}|\nabla\phi(\mathbf{x},t)|^2 + F(\phi(\mathbf{x},t)) \right)\, {\rm d}{\bf x}, \label{fml. free energy}
    \end{align}
    and a gradient flow given by
    \begin{align}
        \frac{\partial\phi}{\partial t} = -\mathcal{G} \mu, \quad \mu = \frac{\delta\mathcal{E}}{\delta \phi} = -\Delta\phi + f(\phi), \quad t\ge 0,\; \mathbf{x}\in\Omega, \label{eq. gradient flow}
    \end{align}
    where $\mathcal{G}$ is a nonnegative symmetric linear operator that determines the dissipation mechanism, $F(\phi)$ is the bulk free energy density, and $f(\phi)=F'(\phi)$. For simplicity, we impose boundary conditions (such as periodic, homogeneous Neumann, or Dirichlet) that do not contribute upon integrating by parts. Applications are made in interface dynamics, crystallization, thin film evolution et al. \cite{Anderson1998Diffuse,  Fraaije2003Model, Leslie1979Theory,Lowengrub1998QuasiIncompressible}.
    
    Defining the $L^2$-inner product by 
    $\langle u,v \rangle := \int_{\Omega} u\, \overline{v}\, {\rm d}\mathbf{x}$, 
    one readily verifies that the gradient flow \eqref{eq. gradient flow} satisfies the energy dissipation law
    \begin{align}
        \frac{{\rm d}\mathcal{E}[\phi]}{{\rm d} t} = \left\langle \frac{\delta\mathcal{E}}{\delta \phi}, \frac{\partial \phi}{\partial t} \right\rangle = -\langle \mu, \mathcal{G} \mu\rangle \le 0, \label{fml. energy stability}
    \end{align} 
    which is a desirable property from a physical point of view, commonly referred to as energy stability. Moreover, many models, such as the Allen–Cahn equation, exhibit a maximum bound principle (MBP),
    \begin{align}
        \|\phi(\cdot,t)\|_{L^{\infty}} \le \beta,\quad \forall\, t\ge 0, \label{fml. boundeness}
    \end{align}
    for some constant $\beta>0$. In cases where $\phi$ is confined to a general interval $[m, M]$, an affine transformation (e.g., $\psi=\frac{\beta}{M-m}(2\phi-M-m)$) will give \eqref{fml. boundeness}. 
    In practical simulations, maintaining energy stability and MBP in the numerical approximation of gradient flows offers significant benefits. A numerical scheme that strictly preserves these properties ensures physical realism by avoiding spurious, nonphysical oscillations or numerical blow-up, and leads to robust long-term behavior in applications such as phase separation and interface evolution.
    
    Over the years, numerous numerical methods have been developed to approximate gradient flows while preserving these nice properties. To preserve energy stability, classical time-stepping schemes have been proposed,  such as convex splitting methods \cite{Chen2019SecondOrder, Elliott1993GlobalDynamics}, stabilized semi-implicit schemes  \cite{ChenWangZhao,Li2017SecondOrder, Shen2015Efficient}. Recently, approaches based on exponential time differencing (ETD) and exponential Runge–Kutta methods \cite{du2021maximum, fu2024higher,Ju2018EnergyStability}, invariant energy quadratization approaches  \cite{Chen2017SecondOrder, Yang2017EfficientLinear, Yang2017Numerical}, and scalar auxiliary variable  approaches \cite{Cheng2020LagrangeMultiplier,Shen2019NewClass} have emerged as attractive alternatives, achieving unconditional energy stability. The efforts have also been devoted to developing numerical methods that preserve MBP. The discrete/numerical MBP could be achieved inherently by specially designed schemes, such as the semi-implicit finite difference \cite{shen2016maximum}, the BDF schemes \cite{hou2023linear}, the ETD schemes \cite{du2019maximum,du2021maximum}, the cutoff approach \cite{li2020arbitrarily} and the Lagrange multiplier approaches \cite{chengshen2022,cheng2022new,van2019positivity}.  It is  a still very challenge  topic for constructing arbitrary high-order schemes in time while preserving energy dissipating and MBP-preserving simultaneously.
    
    In this work, we aim to construct a numerical framework capable of simultaneously preserving both energy stability and MBP while achieving arbitrarily high-order accuracy. The key idea of our new framework is to introduce a new KKT-condition to enforce the energy inequality which can preserve the energy dissipating. The new KKT-condition for energy inequality is formulated by a scalar Lagrange multiplier. Combining with the idea in \cite{chengshen2022,cheng2022new}, we obtain a new expanded system for the original gradient flow system. Based on the expanded system, we shall utilize a novel predictor-corrector-corrector framework, termed the PCC method, which consists of a prediction  from any numerical scheme to the user's favor, followed by two correction steps designed to enforce energy stability and MBP, respectively.  The main contribution of this paper includes:
    \begin{itemize}
\item We introduce a new KKT-condition to enforce the constraint of  energy  inequality and obtain a new equivalent system to construct structure-preserving schemes for gradient flows.
\item We establish a new framework for constructing energy dissipating and MBP-preserving schemes which can be combined with most existing numerical methods, for example, ETD Runge Kutta schemes, BDF schemes.    
\item We give the unique solvability analysis of our numerical schemes by rewriting the correction steps as optimization problems.
\item We give a robust error analysis for the ETD Runge Kutta schemes based on our new framework.
    \end{itemize}
Furthermore, the efficiency and accuracy of the proposed schemes and its structure-preserving properties are also validated by numerical experiments.

    
    The rest of this paper is organized as follows. In Section~\ref{sec. 2}, we present and analyze our framework for constructing energy-stable and MBP-preserving schemes. Section~\ref{sec. 3} gives the explicit and efficient implementation, together with the convergence analysis. Section~\ref{sec. 4} investigates the impact from the spatial discretization, and Section~\ref{sec. 5} presents numerical results. 
   Conclusions are drawn in Section~\ref{sec. 6}.
    
    \section{Energy-stable and MBP-preserving schemes} \label{sec. 2}
    For simplicity of notation, we will omit the spatial variable of a function when it is not necessary, e.g., $f(t)=f(\mathbf{x},t)$. 
    The energy dissipative law in \eqref{fml. energy stability} immediately implies the monotonicity property for \eqref{eq. gradient flow},
    \begin{equation}\label{ineq:en}
    	\mathcal{E}(t_1) \leq \mathcal{E}(t_2), \quad\forall\, t_1 \ge t_2.
    \end{equation}
    The MBP can be read as the constraint $p(\phi) \geq 0$, where we introduce an additional quadratic function 
    $p(\phi) = -(\phi+\beta)(\phi-\beta)$. 
    Our goal is to construct numerical schemes for solving \eqref{eq. gradient flow} that satisfy a discrete analogue of \eqref{ineq:en} and the MBP. To meet both constraints, we introduce two Lagrange multipliers: $\eta$ and $\lambda=\lambda(\mathbf{x})$, and derive the corresponding Karush–Kuhn–Tucker (KKT) conditions (see, e.g., \cite{bergounioux1999primal, facchinei2003finite, harker1990finite, ito2008lagrange}). Consequently, we formulate the original gradient flow system \eqref{eq. gradient flow} as
    \begin{equation}\label{expand:g}
    	\begin{split}
    		&\partial_t \phi = -(\eta + \mathcal{G})\frac{\delta \mathcal{E}}{\delta \phi} + \lambda\, p'(\phi),\\[1mm]
    		&\eta \geq 0,\quad \mathcal{E}(t_1)-\mathcal{E}(t_2)\leq 0,\quad \eta\Bigl(\mathcal{E}(t_1)-\mathcal{E}(t_2)\Bigr)=0,\quad\forall\, t_1 \ge t_2,\\[1mm]
    		&\lambda \geq 0,\quad p(\phi)\geq 0,\quad \lambda\, p(\phi)=0.
    	\end{split}
    \end{equation}
    
    Applying an operator-splitting technique to \eqref{expand:g}, we obtain the following three sub-systems:
\begin{equation}\label{expand:0}
    	\partial_t \phi = -\mathcal{G}\frac{\delta \mathcal{E}}{\delta \phi};
    \end{equation}    
\begin{equation}\label{expand:1}
    	\begin{split}
    		&\partial_t \phi = -\eta\frac{\delta \mathcal{E}}{\delta \phi},\\[1mm]
    		&\eta \geq 0,\quad \mathcal{E}(t_1)-\mathcal{E}(t_2)\leq 0,\quad \eta\Bigl(\mathcal{E}(t_1)-\mathcal{E}(t_2)\Bigr)=0;
    	\end{split}
    \end{equation}
    \begin{equation}\label{expand:2}
    	\begin{split}
    		&\partial_t \phi = \lambda\, p'(\phi),\\[1mm]
    		&\lambda \geq 0,\quad p(\phi)\geq 0,\quad \lambda\, p(\phi)=0.
    	\end{split}
    \end{equation}
  Assuming that \eqref{eq. gradient flow} is in a Banach space $(X,\|\cdot\|)$,  \eqref{expand:1} and \eqref{expand:2} are essentially analogous to two $L^{2}$-projection processes. That is, for any given initial value $\widetilde{\phi}$,  \eqref{expand:1} can be interpreted as the following energy-constrained optimization problem:
    \begin{align*}
    	\min_{\phi\in X} \; \frac{1}{2}\|\phi-\widetilde{\phi}\|^2,\quad \text{s.t.} \quad \mathcal{E}(t_1)-\mathcal{E}(t_2)\leq 0, \quad\forall\, t_1 \ge t_2,
    \end{align*}
    and \eqref{expand:2} can similarly  be viewed as the maximum-bound constrained optimization:
    \begin{align*}
    	\min_{\phi\in X} \; \frac{1}{2}\|\phi-\widetilde{\phi}\|^2,\quad \text{s.t.}\quad -p(\phi)\leq 0.
    \end{align*}
    
    Based on the split systems \eqref{expand:1}--\eqref{expand:2}, we are going to develop numerical schemes that are simultaneously energy-dissipative and maximum-bound preserving. Our approach is based on a predictor-corrector strategy, where an arbitrary conventional numerical method can be used as predictor, followed by two consecutive correction steps (PCC) designed to enforce energy stability and MBP. 
    In what follows, we will detail the construction of the PCC schemes, where we focus on the time discretization here and discuss the spatial discretization later. Rigorous analysis is devoted to establishing the preservation of structures and the solvability of the schemes. 

    \subsection{A class of PCC schemes}
    
    Let $\tau>0$ be the time step size with the time grids $t_{n}=n\tau$ for $n=0,1,\cdots$, and denote $\phi^n\approx\phi(t_n)$ as the numerical solution. 
    Suppose that $\{\phi^{m}\}_{m=0}^{n}$ have been obtained, with the initial value $\phi^0$ given. 
    To proceed to the next time level $t_{n+1}$, we first adopt a consistent numerical integrator $\Phi_\tau(\cdot)$ to the user's favor, such as a semi-implicit finite difference or an ETD Runge–Kutta method, to obtain a prediction of the solution $\phi^{n,s+1}\approx \phi^{n+1}$ (here, the superscript $s$ is to reflect the temporal accuracy of $\Phi_\tau(\cdot)$). This predictor could be interpreted as a numerical integration of the flow \eqref{expand:0} from $t_n$ to $t_{n+1}$, which may not meet the MBP or energy stability. Next, with the two subsequent flows \eqref{expand:1} and \eqref{expand:2} properly discretized, e.g., by finite difference,  the predicted value $\phi^{n,s+1}$ will be corrected through them in a sequel. In such a framework, the detailed \textbf{PCC} scheme reads as follows:
    
    \medskip
    
    {\bf Step 1 (Predictor):} Solve 
    \begin{align}
    	\phi^{n,s+1} = \Phi_\tau(\phi^n), \label{scm. Predictor} 
    \end{align}
    for any consistent integrator $\Phi_\tau$.
    
    \medskip
    
    {\bf Step 2 (Corrector 1 -- energy stability projection):} With $\phi^{n,s+1}$, determine $\phi_{*}^{n+1}$ from
    \begin{subequations}\label{scm. Corrector 1}
    	\begin{align}
    		&\phi_{*}^{n+1} - \phi^{n,s+1} = -\eta^{n+1}\Bigl[-\Delta \phi_{*}^{n+1}+f\bigl(\phi^{n,s+1}\bigr)\Bigr], \label{scm. sub.a Corrector 1}\\[1mm]
    		&\eta^{n+1} \ge 0,\quad \mathcal{E}\bigl[\phi_{*}^{n+1}\bigr]-\mathcal{E}\bigl[\phi^{n}\bigr]\le 0,\quad \eta^{n+1}\Bigl(\mathcal{E}\bigl[\phi_{*}^{n+1}\bigr]-\mathcal{E}\bigl[\phi^{n}\bigr]\Bigr)=0. \label{scm. sub.b Corrector 1}
    	\end{align}
    \end{subequations}
    
    \medskip
    
    {\bf Step 3 (Corrector 2 -- maximum bound projection):} With $\phi_{*}^{n+1}$, compute $\phi^{n+1}$ from
    \begin{subequations}\label{scm. Corrector 2}
    	\begin{align}
    		&\phi^{n+1} - \phi_{*}^{n+1} = \lambda^{n+1}p'\bigl(\phi^{n+1}\bigr), \label{scm. sub.a Corrector 2}\\[1mm]
    		&\lambda^{n+1} \ge 0,\quad p\bigl(\phi^{n+1}\bigr)\ge 0,\quad \lambda^{n+1}p\bigl(\phi^{n+1}\bigr)=0. \label{scm. sub.b Corrector 2}
    	\end{align}
    \end{subequations}
    Here, $\eta^{n+1}$ and $\lambda^{n+1}$ can be interpreted as the scaled (by time step) Lagrange multipliers in \eqref{expand:1} and \eqref{expand:2}.  
    \subsection{Comparisons with existing schemes}
     Before getting into more detailed discussion and analysis, we use the Allen-Cahn equation as a numerical example (for detailed settings, see Example \ref{Ex 1} in Section \ref{sec. 5}), where the exact solution is bounded in interval $[-1,1]$, to quickly demonstrate the efficiency and accuracy of PCC. First, we compare PCC with several existing MBP-preserving schemes, as shown in Fig.~\ref{fig:test_030} and \ref{fig:test_031}. The benchmark schemes are:
    \begin{itemize}
    \item the first-order semi-implicit BDF1 scheme \cite{shen2016maximum};
    \item the exponential time-differencing (ETD) schemes of orders one and two \cite{du2021maximum};
    \item the second-order predictor-corrector BDF2-PC scheme \cite{cheng2022new}.
    \end{itemize}
    The first two methods preserve MBP only under suitable spatial discretization and stabilization parameters (which in the original formulations appear as time-step restrictions), and they fail to maintain the solution bound at machine precision when using Fourier-spectral discretization. The BDF2-PC scheme, while unconditionally MBP-preserving, does not guarantee discrete energy stability.
    
    Next, we solve the Allen-Cahn equation  within or without PCC. 
    Fig.~\ref{fig:test} shows the numerical behavior of the energy (\ref{fml. free energy}) and the extreme values of the numerical solution over time. In the plots, the schemes named with prefix `U' are energy unstable, and those named without suffix `PCC' mean proceeding only as predictor (without corrections). See more explanations in Section \ref{sec. 5.1}. 
     
    As shown in Fig.~\ref{fig:energy_s3}, the plain ETDRK schemes (U-ETDRK3, U-ETDRK4) exhibit strong energy oscillations, while PCC corrections fix the issue and offer energy stability. Figs.~\ref{fig:boundary_s3} and \ref{fig:Lboundary_s3} show the violation of MBP in the plain schemes, and PCC strictly (to machine precision) confines numerical solutions within the exact bounds.  

    \begin{figure}[t!] 
        \centering
        \subfloat[]{%
            \includegraphics[width=0.45\textwidth]{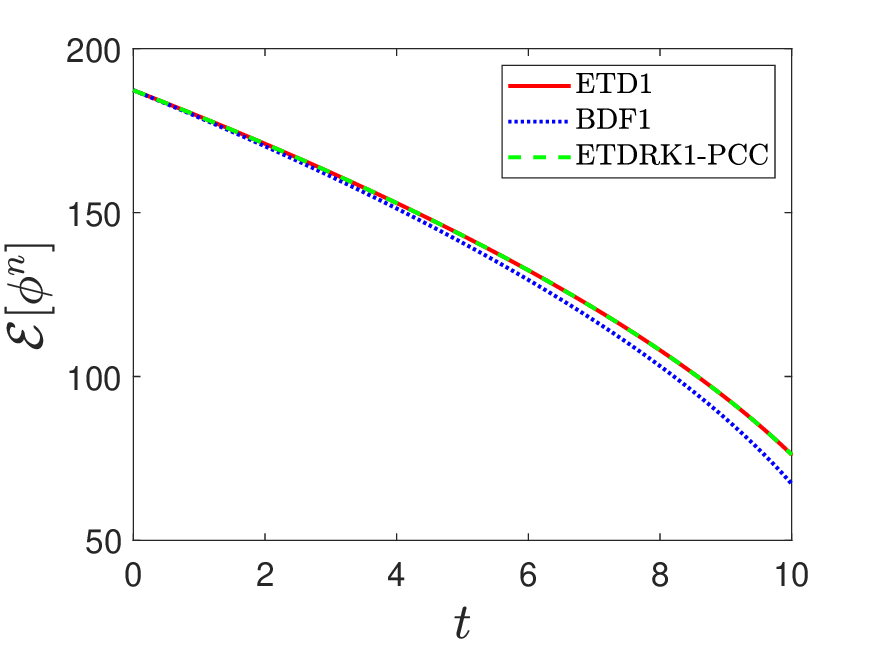} 
            \label{fig:energy_030}
        }
        \subfloat[]{%
            \includegraphics[width=0.45\textwidth]{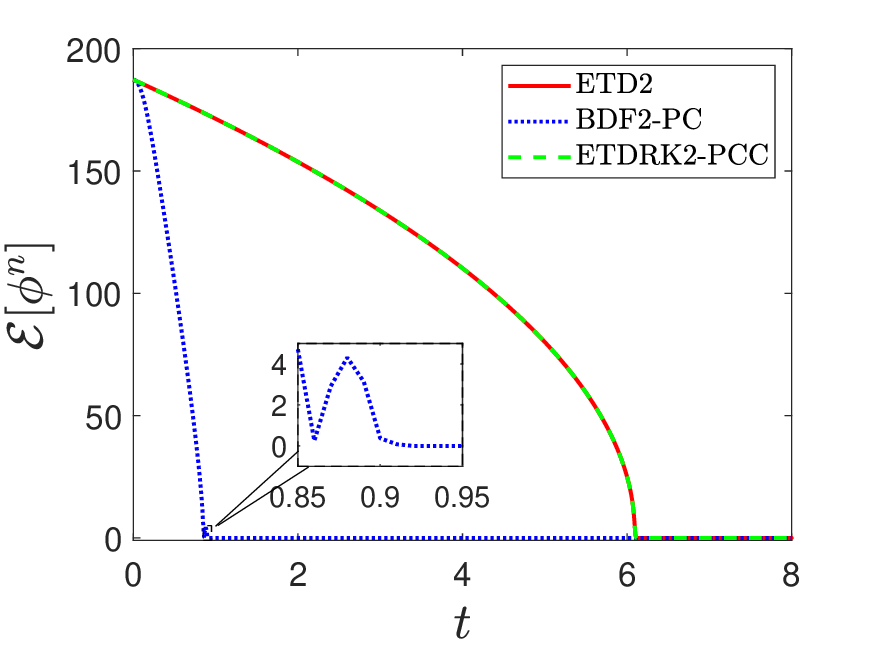} 
            \label{fig:boundary_030}
        }
        \caption{(Example \ref{Ex 1}) Time evolution of energy $\mathcal{E}[\phi^{n}]$ of the numerical solution $\phi^{n}$ with $\tau = 0.01$ and $S=2/\varepsilon^2$.}
        \label{fig:test_030}
    \end{figure}

    \begin{figure}[t!] 
        \centering
        \subfloat[]{%
            \includegraphics[width=0.45\textwidth]{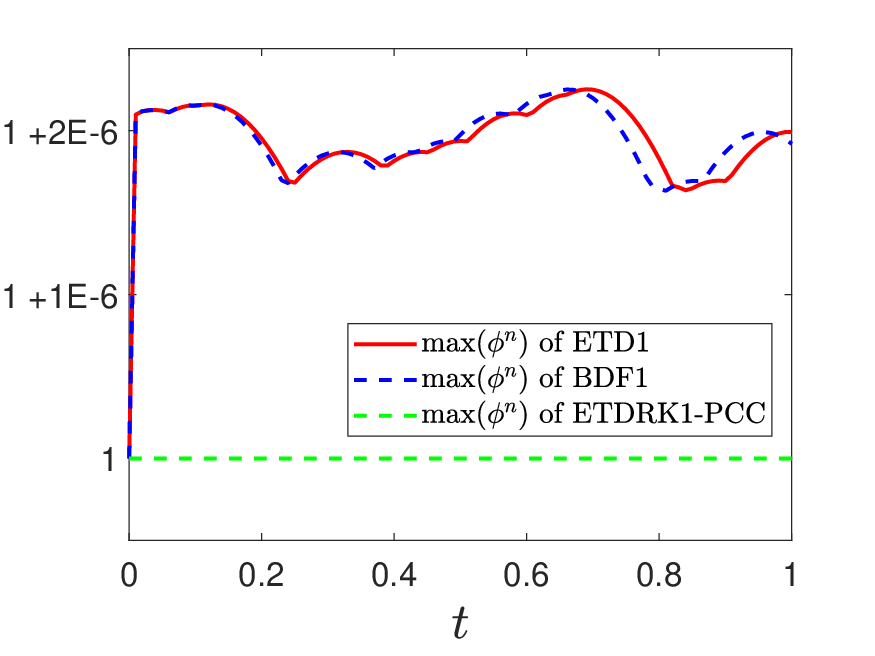} 
            \label{fig:energy_031}
        }
        \subfloat[]{%
            \includegraphics[width=0.45\textwidth]{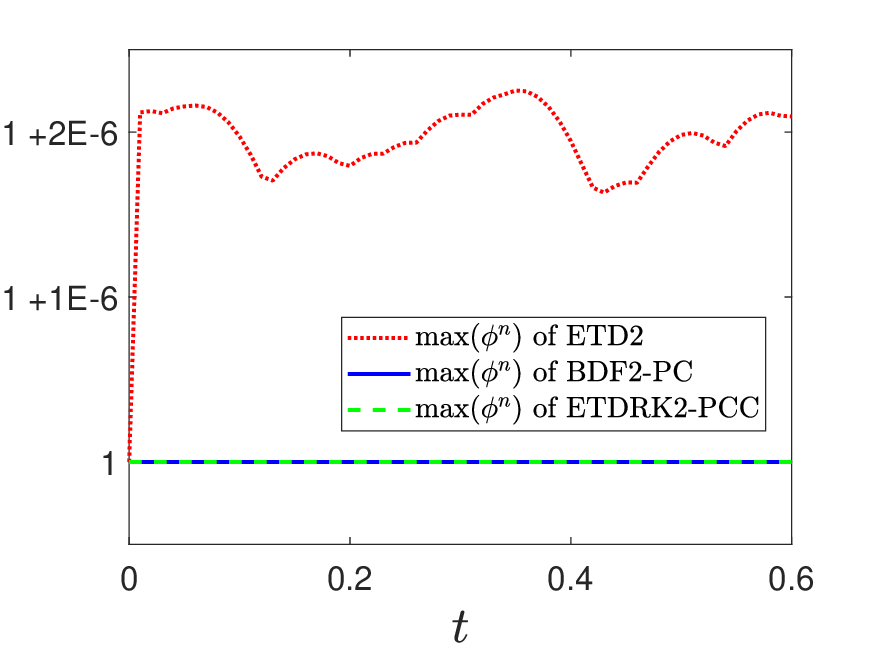} 
            \label{fig:boundary_031}
        }
        \caption{(Example \ref{Ex 1}) The upper bounds of the numerical solution $\phi^{n}$ with $\tau = 0.01$ and $S=2/\varepsilon^2$.}
        \label{fig:test_031}
    \end{figure}
    
    \begin{figure}[t!] 
        \centering
        \subfloat[]{%
            \includegraphics[width=0.31\textwidth]{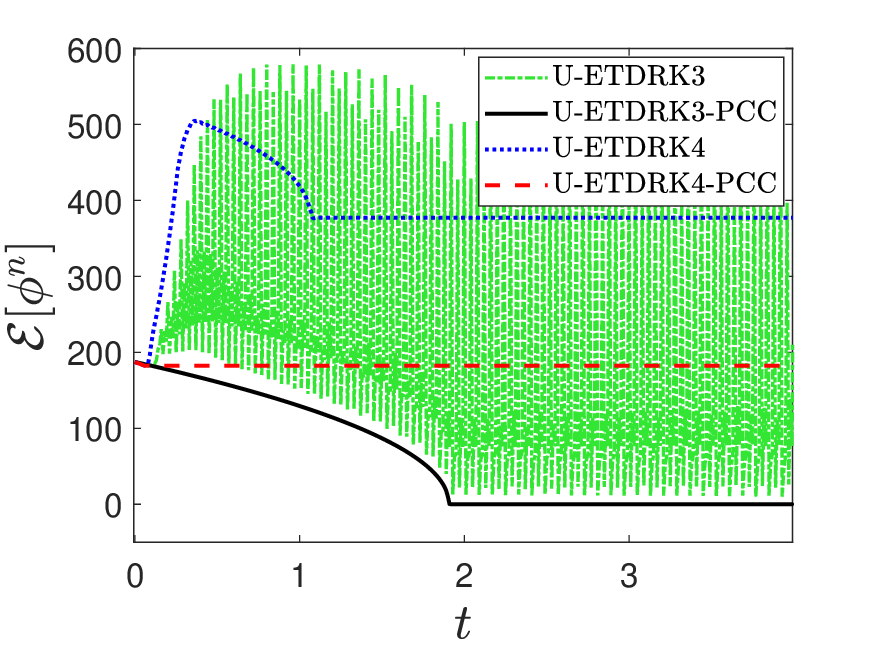} 
            \label{fig:energy_s3}
        }
        \subfloat[]{%
            \includegraphics[width=0.31\textwidth]{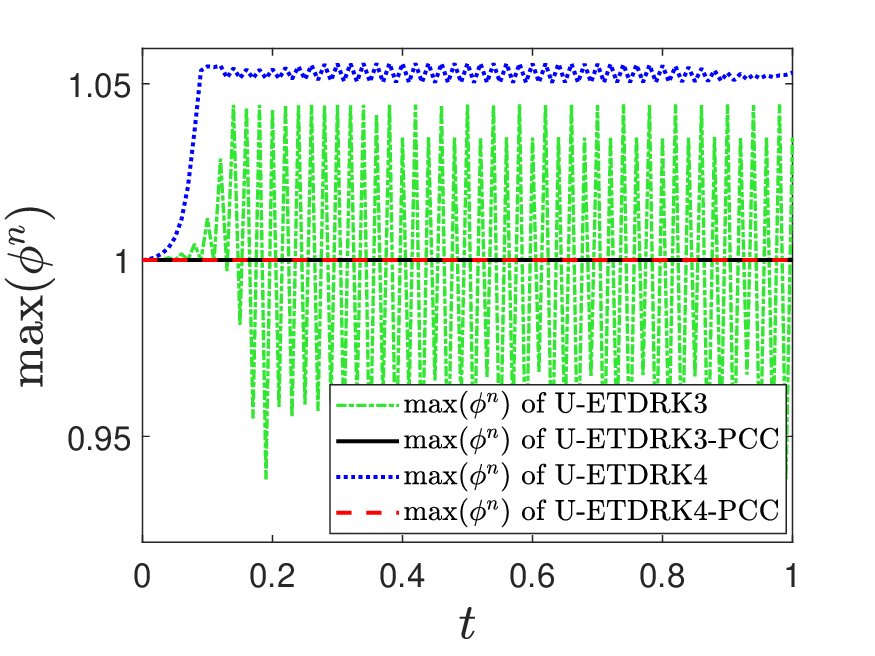} 
            \label{fig:boundary_s3}
        }
        \subfloat[]{%
            \includegraphics[width=0.31\textwidth]{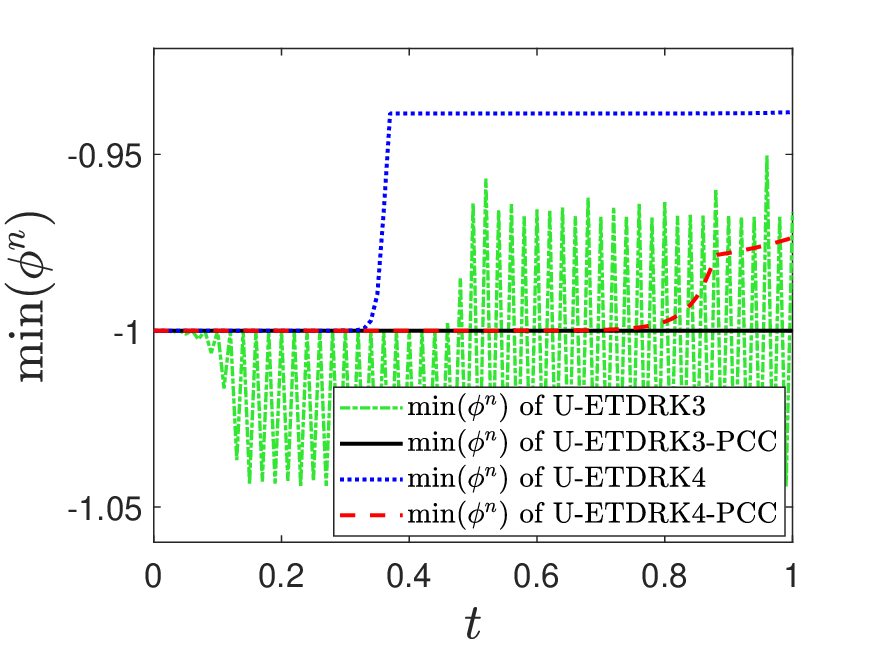} 
            \label{fig:Lboundary_s3}
        } 
        \caption{(Example \ref{Ex 1}) Time evolution of energy $\mathcal{E}[\phi^{n}]$, and the upper and lower bounds of the numerical solution $\phi^{n}$ from U-ETDRK and  U-ETDRK-PCC schemes with $\tau = 0.01$ and $S=1/\varepsilon^2$.}
        \label{fig:test}
    \end{figure}

    Note that the system \eqref{scm. Corrector 1} involved in Step 2 of the PCC algorithm is nonlinear. Here, we provide a practical solver for it. 
    By \eqref{scm. sub.a Corrector 1}, the dependence of $\phi_{*}^{n+1}$ on $\eta^{n+1}$ can be formally expressed as
    $\phi_{*}^{n+1}(\eta^{n+1}) = \left(I-\eta^{n+1}\Delta\right)^{-1}\big[\phi^{n,s+1} - \eta^{n+1}\,f\bigl(\phi^{n,s+1}\bigr)\big]$, where $I$ denotes the identity operator.
    And we define 
    $$D(\eta^{n+1})= \mathcal{E}\bigl[\phi_{*}^{n+1}\bigr]-\mathcal{E}\bigl[\phi^{n}\bigr].$$
    Then, $\eta^{n+1}$ can be determined using the following argument:
    $$
    \begin{array}{lll}
        \mbox{if } D(0)\le 0, & \eta^{n+1}=0, \\[1mm]
        \mbox{if } D(0) > 0, & \eta^{n+1} > 0 \mbox{ such that } D(\eta^{n+1})=0.
    \end{array}
    $$
    In the latter case, $\eta^{n+1}$ is the root of the scalar equation $D(\eta^{n+1})=0$, which can be solved by Newton's method:
    $
    \eta^{n+1}_{k+1} = \eta^{n+1}_{k} - D(\eta^{n+1}_{k})/{D'(\eta^{n+1}_{k})},\ k=0,1,\ldots,
    $
    with a practically effective initial guess $\eta^{n+1}_{0}=0$. 
    The derivative $D'(\eta^{n+1}_{k})$ is computed as \begin{align}
        D'(\eta^{n+1}_k) &= \left\langle \frac{\delta \mathcal{E}\bigl[\phi_{*}^{n+1}\bigr]}{\delta\phi_{*}^{n+1}},\ \frac{\partial \phi_{*}^{n+1}}{\partial \eta^{n+1}}\right\rangle\Big|_{\eta^{n+1}=\eta^{n+1}_k}
        = \Big\langle -\Delta\phi_{*}^{n+1} + f(\phi_{*}^{n+1}), \nonumber\\ & \left( I - \eta_k^{n+1}\Delta \right)^{-2}\Delta\bigl( \phi^{n,s+1} - \eta^{n+1}_k f(\phi^{n,s+1}) \bigr)- \left(I - \eta^{n+1}_k\Delta\right)^{-1}f(\phi^{n,s+1}) \Big\rangle. \label{fml. D diff}
    \end{align}
    
    For Step 3 of PCC, the solution of \eqref{scm. Corrector 2} is in fact a direct cut-off \cite{cheng2022new}, given explicitly as: 
    \begin{align}
        \phi^{n+1} &= 
        \begin{cases}
            -\beta, & \phi_{*}^{n+1} < -\beta, \\
            \phi_{*}^{n+1}, & -\beta \le \phi_{*}^{n+1} \le \beta, \\
            \beta, & \phi_{*}^{n+1} > \beta,
        \end{cases}
        &\lambda^{n+1} =
        \begin{cases}
            \frac{1}{2\beta}\left(-\beta - \phi_{*}^{n+1}\right), & \phi_{*}^{n+1} < -\beta, \\[1mm]
            0, & -\beta \le \phi_{*}^{n+1} \le \beta, \\[1mm]
            \frac{1}{2\beta}\left(\phi_{*}^{n+1} - \beta\right), & \phi_{*}^{n+1} > \beta.
        \end{cases} \label{fml. compute phi star}
    \end{align}

    \subsection{Solvability and structure-preservation}
    We are devoted to analyzing the PCC scheme and we shall begin with its theoretical solvability. To do so, we make the following assumption on the numerical solution (see also \cite{cheng2025unique}).
    
    \begin{assumption}\label{ass. unsteady state}
        We assume that there exists a positive constant $\tilde\epsilon>0$ such that the numerical solution $\phi^{n}$ at $t=n\tau$, $\forall n\leq \frac{T}{\tau}$ satisfies
        \begin{align}
            \left\| -\Delta\phi^{n}+f(\phi^{n}) \right\|^2 \ge \tilde \epsilon. \label{eq:unsteady state}
        \end{align}
    \end{assumption}

    In fact, Assumption~\ref{ass. unsteady state}  may serve as a practical stopping criterion for the PCC scheme. The violation of it indicates that the numerical solution has already reached a steady state.  In such sense, Assumption~\ref{ass. unsteady state} is typically satisfied during the computation.
    
    \begin{theorem}[Existence and uniqueness]\label{thm. existence for ETDRK-PCC}
        Under Assumption~\ref{ass. unsteady state}, there exists a constant $\tau_{1} > 0$ such that for $\tau < \tau_{1}$,  the nonlinear system \eqref{scm. Corrector 1} has a unique solution $\eta^{n+1}$. Moreover, it is always possible to choose $\tau$ sufficiently small so that $\eta^{n+1}$ is arbitrarily small.
    \end{theorem}
    
    To prove Theorem~\ref{thm. existence for ETDRK-PCC}, we utilize the implicit function theorem given below.
    
    \begin{lemma}[\cite{zorich2016mathematical}, Section 10.7, Implicit function theorem]\label{lem. implicit function}
        If a function $D(\tau,\eta)$ satisfies the following conditions:
        \begin{enumerate}
            \item[1)] $D(\tau,\eta)$ is continuous on a region $V \subset \mathbb{R}^2$ that contains $P_{0}=(0, 0)$ as an interior point;
            \item[2)] $D_{\eta}(\tau, \eta):=\partial_\eta D(\tau,\eta)$ is continuous  in $V$;
            \item[3)] $D(0,0) = 0$ and $D_{\eta}(0, 0) \neq 0$.
        \end{enumerate}
        Then, there exist a constant $\tau_1>0$ and  a neighborhood of $P_0$ denoted as $U(P_0) \subset V$ such that the equation $D(\tau, \eta) = 0$ within $U(P_0)$ uniquely defines a continuous function $\eta = \eta(\tau)$ for $\tau \in (-\tau_1, \tau_1)$. Moreover, for each $\tau \in (-\tau_1,\tau_1)$, the point $(\tau, \eta(\tau))$ lies in $U(P_0)$, and we have $D(\tau, \eta(\tau)) = 0$ and $\eta(0) = 0$.
    \end{lemma}
    
    \begin{customproof}[Proof of Theorem~\ref{thm. existence for ETDRK-PCC}]
    Based on \eqref{scm. Predictor} and \eqref{scm. sub.a Corrector 1}, we define
        \begin{align*}
            D(\tau,\eta) := \mathcal{E}\bigl[\psi\bigr]-\mathcal{E}\bigl[\phi^{n}\bigr],\mbox{ with }
            \psi= \left(I-\eta\Delta\right)^{-1}\big[\Phi_\tau(\phi^n) - \eta\,f\bigl(\Phi_\tau(\phi^n)\bigr)\big].
        \end{align*}
        Note when $\tau=\eta=0$, we have  $\psi= \phi^{n}$ and so $D(0,0)=0$. Meanwhile, we can find $D_{\eta}(0,0)$ with the help of (\ref{fml. D diff}): 
        \begin{align*}
            D_{\eta}(0,0) = -\| -\Delta\phi^{n}+f(\phi^{n}) \|^{2} \le -\tilde\epsilon < 0.
        \end{align*}
        Since both $D(\tau,\eta)$ and $D_{\eta}(\tau,\eta)$ are continuous in a neighborhood of $P_0$,  Lemma~\ref{lem. implicit function} guarantees that for  $\tau \in (0,\tau_{1})$, there exists a unique solution $\eta=\eta(\tau)$ such that $D(\tau,\eta)=0$. Furthermore, the fact that $\eta(\tau)$ is continuous and $\eta(0)=0$, means that when choosing $\tau$ sufficiently small, $\eta^{n+1}$ can be arbitrarily small.
    \end{customproof}

    Next, we analyze to rigorously state that PCC is structure-preserving for (\ref{eq. gradient flow}). To this end, we make the following assumption on the nonlinearity of the model, which is typical for guaranteeing MBP \cite{du2021maximum}.  
    \begin{assumption}\label{ass. f boundary}
        Assume that in (\ref{eq. gradient flow}), the nonlinear function $f\in C^{1}$ and there exists a constant $\beta > 0$ such that
        \begin{align*}
            f(z) \le 0, \quad \text{for } z\leq -\beta; \qquad
            f(z) \ge 0, \quad \text{for } z\geq \beta.
        \end{align*}
    \end{assumption}

    \begin{theorem}[Energy stability and MBP preservation]\label{thm. properties for ETDRK-PCC}
        Under  Assumption~\ref{ass. f boundary}, the numerical solution from the PCC scheme (\ref{scm. Predictor}) - (\ref{scm. Corrector 2}) satisfies the following properties for all time levels $n\ge 0$,
        \begin{align*}
            \|\phi^n\|_{L^{\infty}} \le \beta, \quad \text{and} \quad \mathcal{E}[\phi^{n+1}] \le \mathcal{E}[\phi^n].
        \end{align*}
    \end{theorem}


    As given in \eqref{fml. compute phi star}, a cut-off process is essentially applied in Corrector 2 through the truncation function
    \begin{equation}\label{Phi def}
        \Phi(z)=
        \begin{cases}
            -\beta, & z<-\beta, \\
            z, & z\in[-\beta,\beta], \\
            \beta, & z>\beta,
        \end{cases}
    \end{equation}
    and $\phi^{n+1} = \Phi(\phi_{*}^{n+1})$. The function $\Phi$ is continuous, piecewise smooth, and satisfies
    \begin{equation}
        |\Phi(z)|\le \beta,\quad |\Phi(z_1)-\Phi(z_2)|\leq |z_1-z_2|, \quad \forall\, z_1,z_2\in\mathbb{R}. \label{fml. Phi property}
    \end{equation}
        To prove Theorem~\ref{thm. properties for ETDRK-PCC}, we need the following chain rule for weak derivatives in Sobolev spaces.
    
    \begin{lemma}[\cite{gilbarg1977elliptic}, Section~7.4, Theorem~7.8]\label{lem. chain rule}
        Let $\mathcal{D}$ be some distributional derivative in the Sobolev space $W^{1}(\Omega)$ and let $\Phi:\mathbb{R}\to\mathbb{R}$ be piecewise smooth with $\Phi'\in L^{\infty}(\mathbb{R})$. Then for every $u\in W^{1}(\Omega)$, we have $\Phi(u)\in W^{1}(\Omega)$, and if $\Gamma$ denotes the set of points where $\Phi$ fails to be differentiable, the chain rule is given by
        \begin{align*}
            \mathcal{D}\bigl(\Phi(u)\bigr) =
            \begin{cases}
                \Phi'(u) \, \mathcal{D}u, & \text{if } u(x)\notin\Gamma, \\
                0,              & \text{if } u(x)\in\Gamma.
            \end{cases}
        \end{align*}
    \end{lemma}
    
    \begin{customproof}[Proof of Theorem~\ref{thm. properties for ETDRK-PCC}]
        The first part of the assertion, i.e., $\|\phi^n\|_{L^{\infty}} \le \beta$,  immediately follows from the effect of Corrector 2. Thus, it suffices to prove the energy stability, i.e., 
        $\mathcal{E}[\phi^{n+1}] \le \mathcal{E}[\phi^n].$
        From Corrector 1, we naturally have
        \begin{align}
            \mathcal{E}[\phi_{*}^{n+1}] \leq \mathcal{E}[\phi^{n}]. \label{fml. energy corrector 1}
        \end{align}
        As a consequence of Lemma~\ref{lem. chain rule}, for any $\phi_{*}^{n+1}\in H^1(\Omega)$, the corrected solution satisfies
        \begin{align*}
            \phi^{n+1}=\Phi(\phi_{*}^{n+1}) \in H^1(\Omega), \quad \text{with} \quad \nabla\Phi(\phi_{*}^{n+1}) =
            \begin{cases}
                \nabla \phi_{*}^{n+1}, & \text{if } \phi_{*}^{n+1}\in(-\beta,\beta), \\
                0,        & \text{otherwise}.
            \end{cases}
        \end{align*}
        This observation leads to
        \begin{align}
            \left\| \nabla\phi^{n+1} \right\|^2 = \int_{\Omega} |\nabla\Phi(\phi_{*}^{n+1})|^2\, \mathrm{d}\mathbf{x} \leq \int_{\Omega} |\nabla\phi_{*}^{n+1}|^2\, \mathrm{d}\mathbf{x} = \left\| \nabla\phi_{*}^{n+1} \right\|^2. \label{fml. H1 diminish}
        \end{align}
        Furthermore,  Assumption~\ref{ass. f boundary} indicates the monotonicity of $F(\cdot)$ in $(-\infty,-\beta]$ and $[\beta,\infty)$, and so we have
        \begin{align}
            F\bigl(\Phi(\phi_{*}^{n+1})\bigr) \le F\bigl(\phi_{*}^{n+1}\bigr)
            \quad \Longrightarrow \quad \int_{\Omega} F\bigl(\phi^{n+1}\bigr)\, \mathrm{d}\mathbf{x} \le \int_{\Omega} F\bigl(\phi_{*}^{n+1}\bigr)\, \mathrm{d}\mathbf{x}. \label{fml. F diminish}
        \end{align}
        Combining \eqref{fml. H1 diminish} and \eqref{fml. F diminish} yields
        $\mathcal{E}\bigl[\phi^{n+1}\bigr] \le \mathcal{E}\bigl[\phi_{*}^{n+1}\bigr].$
        Substituting this inequality into \eqref{fml. energy corrector 1} completes the proof of the overall energy stability of the PCC scheme.
    \end{customproof}
    
    \subsection{Alternative way of correction}
    Concerning the splitting technique (\ref{expand:0})-(\ref{expand:2}), the composition of sub-flows may be in different orders. 
    In fact, if we swap the order of Corrector 2 and Corrector 1 in the proposed PCC \eqref{scm. Corrector 1}--\eqref{scm. Corrector 2}, we would end up with a new scheme. It still works in the sense that the energy-stability projection would not break the maximum bound, which we shall show in this subsection. The detailed alternative scheme reads:
    
    \textbf{Step 1 (Predictor):} $\phi^{n,s+1}=\Phi_\tau(\phi^{n})$ for any consistent solver $\Phi_\tau$.
    
    \textbf{Step 2 (Corrector 1 -- maximum-bound projection)}: 
    \begin{subequations}\label{scm. Corrector 1 again}
    	\begin{align}
    		&\phi_{*}^{n+1} - \phi^{n,s+1} = \lambda^{n+1} p'(\phi_{*}^{n+1}), \label{scm. sub.a Corrector 1 again} \\[1mm]
    		&\lambda^{n+1} \ge 0,\quad p(\phi_{*}^{n+1}) \ge 0,\quad \lambda^{n+1} p(\phi_{*}^{n+1}) = 0. \label{scm. sub.b Corrector 1 again}
    	\end{align}
    \end{subequations}
    
    \textbf{Step 3 (Corrector 2 -- energy-stable projection)}: 
    \begin{subequations}\label{scm. Corrector 2 again}
    	\begin{align}
    		&\phi^{n+1} - \phi_{*}^{n+1} = -\eta^{n+1}\Bigl[-\Delta\phi^{n+1} + f\bigl(\phi_{*}^{n+1}\bigr)\Bigr], \label{scm. sub.a Corrector 2 again}\\[1mm]
    		&\eta^{n+1} \ge 0,\quad \mathcal{E}\bigl[\phi^{n+1}\bigr] - \mathcal{E}\bigl[\phi^{n}\bigr] \le 0,\quad \eta^{n+1}\Bigl(\mathcal{E}\bigl[\phi^{n+1}\bigr] - \mathcal{E}\bigl[\phi^{n}\bigr]\Bigr) = 0. \label{scm. sub.b Corrector 2 again}
    	\end{align}
    \end{subequations}
    We will refer to the above scheme \eqref{scm. Corrector 1 again}-\eqref{scm. Corrector 2 again} as \textbf{PCC'}.  Steps 2\&3 above can be computed as previously discussed, and the theoretical counterparts of PCC' hold in parallel as PCC.
    
    \begin{theorem}[Properties of PCC']\label{thm. properties for ETDRK-PCC2}
        If Assumption~\ref{ass. unsteady state}  holds, 
        then there exists a constant $\tau_{2} > 0$ such that for $\tau < \tau_{2}$, the nonlinear system \eqref{scm. Corrector 2 again} in PCC' has a unique solution $\eta^{n+1}$. Moreover, if Assumptions~\ref{ass. f boundary} holds, then the numerical solution of PCC' satisfies: for all time levels $n\ge 0$, 
        \begin{align*}
        	\|\phi^n\|_{L^{\infty}} \le \beta, \quad \text{and} \quad \mathcal{E}[\phi^{n+1}] \le \mathcal{E}[\phi^n].
        \end{align*}
    \end{theorem}
    
    The proof needs the following two estimates.
    
    \begin{lemma}[\cite{du2021maximum}]\label{lem. MBP Delta}
        For any $a>0$ and $u \in L^\infty(\Omega)$, we have 
        $$
        \left\| (a I - \Delta)^{-1} u \right\|_{\infty} \leq a^{-1} \| u \|_{\infty},
        $$
        where $I$ is the identity operator.
    \end{lemma}
    
    Lemma~\ref{lem. MBP Delta} holds also for the Laplacian operator $\Delta$ under the second-order central finite difference discretization. Its detailed proof can be found in \cite{du2021maximum}.
    
    \begin{lemma}\label{lem. MBP f}
        Under Assumption~\ref{ass. f boundary}, if $\kappa \geq \left\| f'(\cdot) \right\|_{C[-\beta,\beta]}$ holds for some positive constant $\kappa$, then we have $\left| \kappa\xi - f(\xi) \right| \leq \kappa\beta$, for any $\xi \in [-\beta, \beta]$.
    \end{lemma}
    \begin{proof}
        Set $N(\xi) = \kappa\xi - f(\xi)$, we have
        \begin{align*}
            N'(\xi) = \kappa - f'(\xi) \Rightarrow 0 \leq N'(\xi) \leq 2\kappa,
        \end{align*}
        which means that
        \begin{align*}
            -\kappa\beta \leq -\kappa\beta - f(-\beta) = N(-\beta) \leq N(\xi) \leq N(\beta) = \kappa\beta - f(\beta) \leq \kappa\beta.
        \end{align*}
        This immediately leads to the assertion of Lemma \ref{lem. MBP f}.
    \end{proof}
    
    \begin{customproof}[Proof of Theorem~\ref{thm. properties for ETDRK-PCC2}]
        The existence and uniqueness proof of $\eta^{n+1}$ is fundamentally similar to the proof of Theorem~\ref{thm. existence for ETDRK-PCC}, and when $\tau$ is sufficiently small, $\eta^{n+1}$ can also be arbitrarily small. Therefore, we can conclude that under an appropriate setting of $\tau_{2}$, we have \begin{align}
        	\frac{1}{\eta^{n+1}} \geq \left\| f'(\cdot) \right\|_{C[-\beta,\beta]}. \label{fml. assumption eta}
       	\end{align}
        
        The energy dissipation, $\mathcal{E}[\phi^{n+1}] \le \mathcal{E}[\phi^n]$, is also automatically satisfied after the Corrector 2 of PCC' in \eqref{scm. Corrector 2 again}. So here, we only need to prove that the energy projection maintains the maximum bound after Corrector 1. We can rewrite (\ref{scm. Corrector 2 again}) as
        \begin{align}
            \phi^{n+1} = \Bigl( \frac{1}{\eta^{n+1}}I - \Delta \Bigr)^{-1}\Bigl( \frac{1}{\eta^{n+1}}\phi_{*}^{n+1} - f(\phi_{*}^{n+1}) \Bigr). \label{fml. phi n+1 new}
        \end{align}
        From Lemma~\ref{lem. MBP Delta}, we have
        \begin{align*}
            \left\| \Bigl( \frac{1}{\eta^{n+1}}I -\Delta \Bigr)^{-1} \right\|_{\infty} \leq \eta^{n+1}.
        \end{align*}
        According to Lemma \ref{lem. MBP f} with the condition \eqref{fml. assumption eta}, and the fact $\left\| \phi_{*}^{n+1} \right\|_{\infty} \leq \beta$, we obtain
        \begin{align*}
            \left\| \frac{1}{\eta^{n+1}}\phi_{*}^{n+1} - f(\phi_{*}^{n+1}) \right\|_{\infty} \leq \frac{\beta}{\eta^{n+1}}.
        \end{align*}
        Therefore, from (\ref{fml. phi n+1 new}), we conclude that $\left\| \phi^{n+1} \right\|_{\infty} \leq \beta$, which finishes the proof.
    \end{customproof}
    
    \section{An efficient implementation} \label{sec. 3}
    The PCC framework (\ref{scm. Predictor})-(\ref{scm. Corrector 2}) involves a nonlinear system, which poses a burden for practical computing and challenges convergence analysis. 
    In this section, we provide an efficient implementation by imposing energy stability into the predictor.
    
    \subsection{The ETDRK-PC scheme}\label{sec. 3.1}
    We consider the exponential time differencing Runge–Kutta (ETDRK) method \cite{du2021maximum, fu2024higher} as our predictor. Let us provide a brief introduction of the method. Consider a natural convex splitting of the energy:
    $
    \mathcal{E}[\phi] = \mathcal{E}_{lin}[\phi] - \mathcal{E}_{non}[\phi],
    $
    with
    $$
    \mathcal{E}_{lin}[\phi] = \frac{1}{2}\|\nabla\phi\|^2 + \frac{S}{2}\|\phi\|^2,\qquad 
    \mathcal{E}_{non}[\phi] = -\int_{\Omega}F(\phi)\,{\rm d}\mathbf{x} + \frac{S}{2}\|\phi\|^2,
    $$
   and the gradient flow can be written as
    \begin{equation}
        \frac{\partial\phi}{\partial t} = -\mathcal{G}\mu,\qquad \mu = L\phi - g(\phi),\qquad t > 0,\;\mathbf{x}\in\Omega, \label{eq. gradient flow split}
    \end{equation}
    with
    $L = -\Delta + SI, g(\phi) = -f(\phi) + S\phi$.	Here, $S > 0$ is a stabilizing parameter that enhances numerical stability. 
     Based on the Duhamel formula of \eqref{eq. gradient flow split}, 
    \begin{align*}
        \phi(t_n+\tau) = \fe^{-\tau \mathcal{G}L}\phi(t_n) + \int_{0}^{\tau} \fe^{-(\tau-s)\mathcal{G}L}\mathcal{G}g\bigl(\phi(t_n+s)\bigr){\rm d}s,
    \end{align*}
 the ETDRK method looks for  suitable quadrature formulas to approximate the integral. For example, a straightforward approximation $g\bigl(\phi(t_n+s)\bigr)\approx g(\phi(t_n))$ leads to the first-order scheme (ETDRK1):
    \begin{align}
        \phi^{n+1} = \fe^{-\tau \mathcal{G}L}\phi^n + \left(I-\fe^{-\tau \mathcal{G}L}\right)L^{-1}g(\phi^n). \label{scm. ETD1}
    \end{align}
   An $s$-order scheme  ($s\in\mathbb{N}_+$) is obtained by introducing stage values $\phi^{n,i} \approx \phi(t_n+c_i\tau)$ with 
    $
    c_1 = 0,\  c_i \in (0,1],\  i=2,\dots,s,\  c_{s+1} = 1,
    $
    and proceeding in time as 
    \begin{subequations}\label{fml. ETDRK original}
        \begin{align}
            &\phi^{n,1} = \phi^n, \ 
            \phi^{n,i} = \chi_i(-\tau \mathcal{G}L)\phi^n + \tau\sum_{j=1}^{i-1}a_{i,j}(-\tau \mathcal{G}L)\mathcal{G}g\bigl(\phi^{n,j}\bigr),\ i=2,\ldots,s, \\
           &\phi^{n+1}= \phi^{n,s+1} = \chi(-\tau \mathcal{G}L)\phi^n + \tau\sum_{j=1}^{s}b_{j}(-\tau \mathcal{G}L)\mathcal{G}g\bigl(\phi^{n,j}\bigr).
        \end{align}
    \end{subequations}
    Here $\chi(z) = \mathrm{e}^{z}$, $\chi_i(z) = \chi(c_{i}z)$, and the coefficients $a_{i,j},\,b_j$ are  linear combinations of the exponential functions
    $
    \varphi_{0}(z)=\mathrm{e}^z,\  \varphi_{k+1}(z)=[\varphi_{k}(z)-1/k!]/{z},\  k\ge0,
    $
    which in order to have consistency, must satisfy
    \begin{equation}
        \sum_{j=1}^{s} a_{i,j}(z) = \frac{\chi_i(z)-1}{z},\quad \sum_{j=1}^{s} b_j(z) = \frac{\chi(z)-1}{z}. \label{eq. RK equilibria}
    \end{equation}
    In short, the ETDRK schemes can be represented by a Butcher tableau:
    $$
    \begin{array}{c|cccc|c}
        c_{1}  &   &   &   &   & \chi_{1}(z) \\
        c_{2}  & a_{2,1} &   &   &   & \chi_{2}(z)   \\
        \vdots & \vdots & \vdots &   &   & \vdots   \\
        c_{s}  & a_{s,1} & \cdots & a_{s,s-1}  &   & \chi_{s}(z)   \\
        \hline
        & b_{1} & b_{2} & \cdots & b_{s} & \chi(z)
    \end{array}
    $$
    
    Generally speaking, ETDRK does not guarantee energy stability or MBP. 
    The main reasons for us to adopt ETDRK in the prediction step include its high computational efficiency and the possibility of having energy stability by properly constructing the Butcher tableau.  A sufficient condition for energy stability has been given in \cite{fu2024higher} and is summarized as follows.

    \begin{assumption}\label{ass. A}
        Let $E_{L}=(1_{i\ge j})_{s\times s}$ denote an $s\times s$ lower triangular matrix with all nonzero entries equal to $1$, and define the matrix function
        \small
        \begin{align*}
            A(z)=\begin{pmatrix}
                a_{2,1} &         &        &          \\
                a_{3,1} & a_{3,2} &        &          \\
                \vdots  & \vdots  & \ddots &          \\
                a_{s,1} & a_{s,2} & \cdots & a_{s,s-1} \\[1mm]
                b_{1}  & b_{2}   & \cdots & b_{s}
            \end{pmatrix}.
        \end{align*}
        \normalsize
       Assume that the coefficients $a_{i,j},\,b_j$ of ETDRK make 
            $P(z)=zE_{L}+(A(z))^{-1}E_{L}-\frac{z}{2}I$
         positive-definite for all $z<0$.
    \end{assumption}
    
    So far, energy-stable ETDRK schemes that comply with Assumption \ref{ass. A} have been constructed up to third-order accuracy ($s=3$) \cite{fu2024higher}.       Detailed ETDRK schemes that satisfy Assumption~\ref{ass. A} will be given and used in Section \ref{sec. 5}.  Together with the following technical requirement, the rigorous energy stability of ETDRK is stated in Lemma \ref{lem. energy stability of ETDRK}.

    \begin{assumption}\label{ass. f Lipschitz}
        Assume that the nonlinear function $f$ is Lipschitz continuous with Lipschitz constant $C_L$, that is, $\|f(u)-f(v)\| \le C_L\|u-v\|,\  \forall\, u,v\in L^2(\Omega).$ 
    \end{assumption}
    \begin{remark}
       For real applications, the Lipschitz continuity of $f$ may be too strong. Instead, $f\in C^1(\mathbb{R})$ together with the MBP of (\ref{eq. gradient flow}) would be enough.    
    \end{remark}
    \begin{lemma}[\cite{fu2024higher}] \label{lem. energy stability of ETDRK} 
        Suppose that $\mathcal{G}$ is nonnegative, Assumptions~\ref{ass. A} and \ref{ass. f Lipschitz} hold, and the chosen stabilizing parameter $S\ge C_L$. Then, the energy of ETDRK is unconditionally dissipative: $\mathcal{E}[\phi^{n,s+1}]\le \mathcal{E}[\phi^n]$ for all $n\ge0$, with  $\phi^{n,s+1}$  from (\ref{fml. ETDRK original}).
    \end{lemma}
    
    Thanks to the energy stable ETDRK, we can omit Step 2 of PCC, which simplifies the process into the following \textbf{ETDRK-PC} scheme:
    
    \textbf{Step 1 (Predictor):} Compute the prediction $\phi^{n,s+1}$ via an energy-stable and consistent ETDRK scheme (\eqref{fml. ETDRK original} with \eqref{eq. RK equilibria}):
    \begin{subequations}\label{scm. ETDRK Predictor simplified}
        \begin{align}
            &\phi^{n,1} = \phi^n, \ 
            \phi^{n,i} = \phi^n + \tau\sum_{j=1}^{i-1}a_{i,j}(-\tau \mathcal{G}L)
            \Bigl[\mathcal{G}g\bigl(\phi^{n,j}\bigr) - \mathcal{G}L\phi^n\Bigr],\quad i=2,\ldots,s, \\
            &\phi^{n,s+1} = \phi^n + \tau\sum_{j=1}^{s}b_{j}(-\tau \mathcal{G}L)
            \Bigl[\mathcal{G}g\bigl(\phi^{n,j}\bigr) - \mathcal{G}L\phi^n\Bigr],
        \end{align}
    \end{subequations}
    where $a_{i,j},b_j$ satisfy Assumption \ref{ass. A}.
    
    \textbf{Step 2 (Corrector -- maximum-bound projection)}: 
    \begin{subequations}\label{scm. Corrector 1 simplified}
        \begin{align}
            &\phi^{n+1} - \phi^{n,s+1} = \lambda^{n+1} \, p'(\phi^{n+1}), \label{scm. sub.a Corrector 1 simplified} \\[1mm]
            &\lambda^{n+1} \ge 0,\quad p(\phi^{n+1}) \ge 0,\quad \lambda^{n+1} p(\phi^{n+1}) = 0. \label{scm. sub.b Corrector 1 simplified}
        \end{align}
    \end{subequations}
  
    \begin{theorem}[Energy stability and maximum-bound preservation]\label{thm. energy stability}
        The numerical solution $\phi^{n}$ generated by the ETDRK-PC scheme \eqref{scm. ETDRK Predictor simplified}--\eqref{scm. Corrector 1 simplified} satisfies
        \begin{align*}
            \|\phi^{n}\|_{L^{\infty}}\le \beta,\quad \text{for all}\ n\ge0.
        \end{align*}
        Moreover, under Assumptions~\ref{ass. f boundary} and the conditions of Lemma \ref{lem. energy stability of ETDRK}, the energy is unconditionally dissipative:
        \begin{align*}
            \mathcal{E}[\phi^{n+1}]\le \mathcal{E}[\phi^n],\quad \text{for all}\ n\ge0.
        \end{align*}
    \end{theorem}
    
    The proof of Theorem~\ref{thm. energy stability} goes similarly as  Theorem~\ref{thm. properties for ETDRK-PCC}, and is therefore omitted.
    
    \subsection{Convergence analysis}
    Now, we provide a convergence analysis for the ETDRK-PC scheme.
    For simplicity, we shall only consider the case  $\mathcal{G} = I$ (i.e., the Allen-Cahn equation) with periodic or Dirichlet boundary conditions.

    
    
    
    \begin{assumption}\label{ass. phi}
        Assume that the gradient flow \eqref{eq. gradient flow} admits a sufficiently regular solution $\phi$ with the required time and space derivatives, and $f$ is repeatedly Fr\'echet differentiable in a strip along the solution. Furthermore, all relevant derivatives are uniformly bounded.
    \end{assumption}
    
    For high-order accurate methods, certain order conditions must be met. Table~\ref{tab. order condition} lists the required conditions for orders from one to three, where 
    \begin{align*}
        \psi_{j}(z) &= \varphi_{j}(z)-\sum_{k=1}^{s}b_{k}(z)\frac{c_{k}^{j-1}}{(j-1)!},\ \ 
        \psi_{j,i}(z) = c_{i}^{j}\varphi_{j}(-c_{i}z)-\sum_{k=1}^{i-1}a_{i,k}(z)\frac{c_{k}^{j-1}}{(j-1)!},
    \end{align*}
    and $J$ denotes arbitrary bounded operator on $X$. We are now ready to state the convergence result. 
    
    \begin{table}[!ht]
        \caption{Order conditions for the ETDRK methods.}
        \renewcommand{\arraystretch}{1.5}
        \centering
        \begin{tabular}{|c|>{\centering\arraybackslash}m{10cm}|}
            \hline
            Order & Order Condition \\ \hline
            1 & \begin{minipage}[c]{\linewidth}
            		$$\psi_{1}(-\tau L)=0,$$
            \end{minipage} \\ \hline
            2 & \begin{minipage}[c]{\linewidth}
            	$$\psi_{2}(-\tau L)=0,\ \psi_{1,2}(-\tau L)=0$$
            \end{minipage} \\ \hline
            3 & \begin{minipage}[c]{\linewidth}
                $$\psi_{3}(0)=0,\ \psi_{1,3}(-\tau L)=0, \  \sum_{i=2}^{3} b_{i}(0)\, J\, \psi_{2,i}(-\tau L)=0.$$
            \end{minipage} \\ \hline
        \end{tabular}
        \label{tab. order condition}
    \end{table}

    \begin{theorem}[Error Bound]\label{thm. convergency}
        Assume that the gradient flow \eqref{eq. gradient flow} (with $\mathcal{G} = I$) satisfies Assumptions~\ref{ass. f Lipschitz} and~\ref{ass. phi}. Let $\phi^{n+1}$ denote the numerical solution obtained via the ETDRK-PC scheme \eqref{scm. ETDRK Predictor simplified}--\eqref{scm. Corrector 1 simplified}. Moreover, for any $1 \leq s \leq 3$, assume that the order conditions in Table~\ref{tab. order condition} up to order $s$ are satisfied. Then there exists a constant $\tau_{3}>0$, independent of the discrete time level $n$, such that for any time step $0 < \tau \leq \tau_{3}$ the following error bound holds:
        $$\|\phi^{n}-\phi(t_{n})\| \le C\,\tau^{s},\quad 0\le n\,\tau\le T,$$
        where the constant $C>0$ depends solely on $T$ and is independent of $n$ or $\tau$.
    \end{theorem}
    
    \begin{proof}
        The proof is divided into two parts. In the first part, we analyze the error introduced by the ETDRK predictor. In the second part, we study the error arising from the correction step. For brevity, we introduce the notation $t_{n,i}=t_n+c_i\tau$ and write $u=v+\mathcal{O}(\tau^{i})$ to indicate that $\|u-v\|\le C\,\tau^{i}$, where $C$ is a constant independent of $\tau$ or $n$. In what follows, we prove only for the case $s=3$, and the other two cases can be done similarly.
        
        \medskip
        
        \textbf{Step 1. Uniform Boundedness and Local Truncation Error.} 
        First, we prove that the stage values $\phi^{n,i}$ (for $i=1,\ldots,s$) remain uniformly bounded. Note from Theorem \ref{thm. energy stability} that we have $\|\phi^{n}\|_{L^{\infty}} \le \beta$,
        and by Assumption~\ref{ass. phi}, for any function $u$ satisfying $\|u\|_{L^{\infty}}\le\beta+1$, there exists a constant $M>0$ such that
        $\|g(u)\|_{L^{\infty}} \le M$.
        
        Let $v^{n+1} = \mathrm{e}^{\tau\Delta} \phi^{n}$,
        which is the solution of the following parabolic equation at time $\tau$,
        $$
        \partial_t v(\mathbf{x},t) = \Delta v(\mathbf{x},t),\quad
        v(\mathbf{x},0) = \phi^{n}(\mathbf{x}).
        $$
        According to the maximum principle, the maximum value of the solution occurs either on the boundary or initially on $\phi^{n}$. For Dirichlet boundary conditions, the boundary data naturally ensure that $v^{n+1}$ is bounded, whereas for periodic boundary conditions, the maximum is in fact attained only on the initial layer $\phi^{n}$. Consequently, we have  $\|v^{n+1}\|_{L^{\infty}}\le\|\phi^{n}\|_{L^{\infty}} \le \beta$, implying that
        \begin{equation}\label{infty_estimate1}
            \|\mathrm{e}^{-\tau L}\phi^{n}\|_{L^{\infty}} \le \|\mathrm{e}^{-\tau S}\mathrm{e}^{\tau\Delta}\phi^{n}\|_{L^{\infty}} \le \|\phi^{n}\|_{L^{\infty}} \le \beta.
        \end{equation}
        Moreover, as the $\varphi$--functions satisfy
        $\varphi_{k}(-\tau L) = \frac{1}{\tau^{k}} \int_{0}^{\tau} \mathrm{e}^{-(\tau-\theta)L}\frac{\theta^{k-1}}{(k-1)!}\,\mathrm{d}\theta,$
        we see for any $u$ that
        $\|\varphi_{k}(-\tau L)u\|_{L^{\infty}} \le \|u\|_{L^{\infty}}$.
        Since the coefficient functions $a_{i,j}(-\tau L)$ are linear combinations of $\{\varphi_{k}\}$, it follows that for any $\|u\|_{L^{\infty}}\le \beta+1$,
        \begin{equation}\label{infty_estimate2}
            \|a_{i,j}(-\tau L)g(u)\|_{L^{\infty}}\le \|g(u)\|_{L^{\infty}} \le M.
        \end{equation}
        Hence, with $\tau M\le 1$, one obtains from \eqref{fml. ETDRK original} the bounds:
        $$\|\phi^{n,1}\|_{L^{\infty}} \le\beta, \quad \|\phi^{n,2}\|_{L^{\infty}} \le \beta+\tau M, \quad \|\phi^{n,3}\|_{L^{\infty}} \le \beta+2\tau M.$$
        Choosing $\tau_{3}$ so that $2\tau_{3} M\le 1$, we can deduce for $\tau\leq\tau_3$ the uniform boundedness:
        \begin{equation}\label{infty_uniform_boundedness}
            \|\phi^{n,i}\|_{L^{\infty}} \le \beta+1,\quad i=1,2,3.
        \end{equation}
        
        Next, we consider the local truncation error of the ETDRK predictor. Denote $g(t)= g\bigl(\phi(t)\bigr)$. The Duhamel formula of \eqref{eq. gradient flow split} reads
        \begin{equation}\label{proof_01}
            \phi(t_{n,i}) = \mathrm{e}^{-c_i\tau L}\phi(t_n) + \int_{0}^{c_i\tau} \mathrm{e}^{-(c_i\tau-\zeta)L}g(t_n+\zeta)\,\mathrm{d}\zeta.
        \end{equation}
        Expanding $g(t_{n}+\zeta)$ via Taylor's formula yields
        \begin{equation}\label{proof_02}
            g(t_n+\zeta) = \sum_{j=1}^{q}\frac{\zeta^{j-1}}{(j-1)!}g^{(j-1)}(t_n) + \int_{0}^{\zeta} \frac{(\zeta-\sigma)^{q-1}}{(q-1)!}g^{(q)}(t_n+\sigma)\,\mathrm{d}\sigma,
        \end{equation}
        where $q>0$ is an integer to be chosen appropriately later. Substituting \eqref{proof_02} into \eqref{proof_01} leads to
        \begin{equation}\label{proof_03}
            \begin{split}
                \phi(t_{n,i}) =\,& \chi_i(-\tau L)\phi(t_n) + \sum_{j=1}^{q_i}(c_i\tau)^{j}\varphi_{j}(-c_i\tau L)g^{(j-1)}(t_n)\\[1mm]
                &+ \int_{0}^{c_i\tau} \mathrm{e}^{-(c_i\tau-\zeta)L} \int_{0}^{\zeta}\frac{(\zeta-\sigma)^{q_i-1}}{(q_i-1)!} g^{(q_i)}(t_n+\sigma)\,\mathrm{d}\sigma\,\mathrm{d}\zeta.
            \end{split}
        \end{equation}
        Based on the numerical scheme \eqref{fml. ETDRK original}, we define the local truncation error $R_{n,i}$ for $i=1,\ldots,s+1$ via
        \begin{subequations}\label{local_truncation}
            \begin{align}
                \phi(t_{n,i}) &= \chi_i(-\tau L)\phi(t_n) + \tau\sum_{k=1}^{i-1}a_{i,k}(-\tau L)g(t_{n,k}) + R_{n,i}, \label{local_truncation1}\\[1mm]
                \phi(t_{n+1}) &= \chi(-\tau L)\phi(t_n) + \tau\sum_{i=1}^{s}b_{i}(-\tau L)g(t_{n,i}) + R_{n,s+1}. \label{local_truncation2}
            \end{align}
        \end{subequations}
        Performing the Taylor expansion (\ref{proof_02}) of $g(t_{n,k})$ in \eqref{local_truncation1}, and subsequently subtracting \eqref{proof_03} from \eqref{local_truncation1}, it can be observed that for $i = 1, \ldots, s$,
        $$R_{n,i} = \sum_{j=1}^{q_i}\tau^{j}\psi_{j,i}(-\tau L)g^{(j-1)}(t_n) + \mathcal{O}(\tau^{q_i+1}).$$
        Similarly, one can find
        $$R_{n,s+1} = \sum_{j=1}^{q_{s+1}}\tau^{j}\psi_{j}(-\tau L)g^{(j-1)}(t_n) + \mathcal{O}(\tau^{q_{s+1}+1}).$$
        
        Let $E^n = \phi^{n} - \phi(t_n)$ and $E^{n,i} = \phi^{n,i} - \phi(t_{n,i})$ denote the errors at the discrete time levels $t_n$ and $t_{n,i}$, respectively. Subtracting \eqref{local_truncation1}--\eqref{local_truncation2} from the corresponding numerical step \eqref{fml. ETDRK original}, we obtain
        \begin{subequations}\label{error_eqs}
            \begin{align}
                E^{n,1} =& E^n, \label{error_eq0}\\[1mm]
                E^{n,i} =& \chi_i(-\tau L)E^n + \tau\sum_{j=1}^{i-1}a_{i,j}(-\tau L)\Bigl( g\bigl(\phi^{n,j}\bigr) - g\bigl(\phi(t_{n,j})\bigr) \Bigr)\label{error_eq1}\\
                &- R_{n,i}, \quad i=2,\ldots,s,\nonumber \\[1mm]
                E^{n,s+1} =& \chi(-\tau L)E^n + \tau\sum_{i=1}^{s}b_{i}(-\tau L)\Bigl( g\bigl(\phi^{n,i}\bigr) - g\bigl(\phi(t_{n,i})\bigr) \Bigr) - R_{n,s+1}. \label{error_eq2}
            \end{align}
        \end{subequations}
        To estimate the nonlinear terms in \eqref{error_eqs}, we apply a Taylor expansion with an integral remainder for any $u\in L^2(\Omega)$. In particular, for any given stage $t_{n,i}$, we have
        \begin{equation}\label{proof_definite_J}
            \begin{split}
                &g(\phi(t_{n,i}) + u) - g(\phi(t_{n,i})) \\
                =& \frac{\partial g}{\partial\phi}\bigl(\phi(t_{n,i})\bigr)u + \int_{0}^{1}(1-\theta)\frac{\partial^{2}g}{\partial\phi^{2}}\Bigl(\phi(t_{n,i}) + \theta u\Bigr)\,\mathrm{d}\theta\;\cdot (u)^2\\[1mm]
                =& \frac{\partial g}{\partial\phi}\bigl(\phi(t_n)\bigr)u + \left(\phi(t_{n,i})-\phi(t_n)\right)\int_{0}^{1}\frac{\partial^{2}g}{\partial\phi^{2}}\Bigl((1-\theta)\phi(t_n)+\theta\phi(t_{n,i})\Bigr)\,\mathrm{d}\theta\cdot u\\[1mm]
                &+ \int_{0}^{1}(1-\theta)\frac{\partial^{2}g}{\partial\phi^{2}}\Bigl(\phi(t_{n,i})+\theta u\Bigr)\,\mathrm{d}\theta\; \cdot (u)^2,
            \end{split}
        \end{equation}
        where the first term is evaluated at $\phi(t_n)$. Along with the remaining terms, we define
        $$J_n := \frac{\partial g}{\partial\phi}\bigl(\phi(t_n)\bigr),\quad K_{n,i} := \left(\phi(t_{n,i})-\phi(t_n)\right)\int_{0}^{1}\frac{\partial^{2}g}{\partial\phi^{2}}\Bigl((1-\theta)\phi(t_n)+\theta\phi(t_{n,i})\Bigr)d\theta,$$
        and introduce the operator
        $$Q_{n,i}(u) := \int_{0}^{1}(1-\theta)\frac{\partial^{2}g}{\partial\phi^{2}}\Bigl(\phi(t_{n,i})+\theta u\Bigr)d\theta\; \cdot u.$$
        Thus, we may denote
        \[
        J_{n,i}(u):= J_n + K_{n,i} + Q_{n,i}(u).
        \]
        By the uniform bound \eqref{infty_uniform_boundedness}, the Lipschitz continuity of $f$ (Assumption~\ref{ass. f Lipschitz}) and the sufficient smoothness of $g$ (Assumption~\ref{ass. phi}), we deduce that
        \begin{subequations}\label{operator_J}
            \begin{align}
                &\|J_n u\| \le C\|u\|, \quad \|K_{n,i}u\| \le C\,\tau\|u\|,\\[1mm]
                &\|Q_{n,i}(E^{n,i})u\| \le C\|E^{n,i}\|_{L^{\infty}}\|u\| \le C\|u\|, \quad \|J_{n,i}(E^{n,i})u\| \le C\|u\|.
            \end{align}
        \end{subequations}
        
       Now for $i=2$ in \eqref{error_eqs}, from \eqref{error_eq0} and \eqref{error_eq1}, one obtains
        \begin{equation}\label{proof_En2}
            \begin{split}
                E^{n,2} =& \chi_2(-\tau L)E^n + \tau\,a_{2,1}(-\tau L)\Bigl( J_{n,1}(E^n)E^n \Bigr) - R_{n,2} \\
                =:& \chi_2(-\tau L)E^n + \tau\,\mathcal{N}_{n,2} - R_{n,2}.
            \end{split}
        \end{equation}
        Similarly, the error equation for $i=3$ reads
        \begin{equation}\label{recursive_En3}
            \begin{split}
                E^{n,3} =&\, \chi_3(-\tau L)E^n + \tau\sum_{j=1}^{2} a_{3,j}(-\tau L)\Bigl( J_{n,j}(E^{n,j})E^{n,j} \Bigr) - R_{n,3}\\[1mm]
                =&\, \chi_3(-\tau L)E^n + \tau\,\mathcal{N}_{n,3}(E^n) - \tau\,a_{3,2}(-\tau L)\Bigl( J_{n,2}(E^{n,2})R_{n,2} \Bigr) - R_{n,3},
            \end{split}
        \end{equation}
        where we have defined the operator
        \begin{align*}
            \mathcal{N}_{n,3}(E^n) =& a_{3,1}(-\tau L)\Bigl( J_{n,1}(E^n)E^n \Bigr) \\
            &+ a_{3,2}(-\tau L)\Bigl[ J_{n,2}(E^{n,2})\Bigl( \chi_2(-\tau L)E^n + \tau\,\mathcal{N}_{n,2}(E^n) \Bigr) \Bigr].
        \end{align*}
        Choosing $q_2 = 1$ and noting that the order condition implies $\psi_{1,2}(-\tau L)=0$, we find
        $-\tau\,a_{3,2}(-\tau L)\bigl( J_{n,2}(E^{n,2})R_{n,2} \bigr) = \mathcal{O}(\tau^3),$
        and so
        \begin{equation}\label{proof_En3_final}
            E^{n,3} = \chi_3(-\tau L)E^n + \tau\,\mathcal{N}_{n,3}(E^n) - R_{n,3} + \mathcal{O}(\tau^3).
        \end{equation}
        Furthermore, we can deduce from \eqref{error_eq2} that
        \begin{equation}\label{recursive_En4}
            \begin{split}
                E^{n,4} =&\, \chi(-\tau L)E^n + \tau \sum_{i=1}^{3} b_{i}(-\tau L) \Bigl( J_{n,i}(E^{n,i})E^{n,i} \Bigr) - R_{n,4}\\[1mm]
                =&\, \chi(-\tau L)E^n + \tau\,\mathcal{N}_{n,4}(E^n) - \tau\sum_{i=2}^{3}b_{i}(-\tau L)\Bigl( J_{n,i}(E^{n,i})R_{n,i} \Bigr) - R_{n,4},
            \end{split}
        \end{equation}
        where
        \begin{align*}
            \mathcal{N}_{n,4}(E^n) =& b_{1}(-\tau L) \Bigl( J_{n,1}(E^n)E^n \Bigr) \\
            &+ \sum_{i=2}^{3}b_{i}(-\tau L)\Bigl[ J_{n,i}(E^{n,i})\Bigl( \chi_i(-\tau L)E^n + \tau\,\mathcal{N}_{n,i}(E^n) \Bigr) \Bigr].
        \end{align*}
        Again, by choosing $q_2=q_3=2$ (so that $\psi_{1,i}(-\tau L)=0$) and $q_4 = 3$, one obtains
        \begin{align*}
            \begin{split}
                E^{n,4} =& \chi(-\tau L)E^{n} + \tau\mathcal{N}_{n,4}(E^{n})-\sum_{i=1}^{3}\tau^{i}\psi_{i}(-\tau L)g^{(i-1)}(t_{n}) + \mathcal{O}(\tau^{4}) \\
                &- \tau^{3}\sum_{i=2}^{3}b_{i}(-\tau L)\left( J_{n,i}(E^{n,i})\psi_{2,i}(-\tau L)g^{(1)}(t_{n}) \right).
            \end{split}
        \end{align*}
        For the last term above, we have by the definition of $J_{n,i}(E^{n,i})$,
        \begin{align*}
            & -\tau^{3}\sum_{i=2}^{3}b_{i}(-\tau L)\left( J_{n,i}(E^{n,i})\psi_{2,i}(-\tau L)g^{(1)}(t_{n}) \right) \\
            =& -\tau^{3}\sum_{i=2}^{3}b_{i}(-\tau L)\left( \left(J_{n}+K_{n,i}+Q_{n,i}(E^{n,i})\right)\psi_{2,i}(-\tau L)g^{(1)}(t_{n}) \right) \\
            =& -\tau^{3}\sum_{i=2}^{3}b_{i}(-\tau L)\left( \left(J_{n}+Q_{n,i}(E^{n,i})\right)\psi_{2,i}(-\tau L)g^{(1)}(t_{n}) \right) + \mathcal{O}(\tau^{4}).
        \end{align*}
        Combining (\ref{proof_En2}) and (\ref{proof_En3_final}), we have
        \begin{align*}
            & \tau^{3}\left\| \sum_{i=2}^{3}b_{i}(-\tau L)\left( Q_{n,i}(E^{n,i})\psi_{2,i}(-\tau L)g^{(1)}(t_{n}) \right) \right\| \\
            \leq& C\tau^{3}\sum_{i=2}^{3}\|E^{n,i}\| \leq C\tau^{3}\|E^{n}\| + \mathcal{O}(\tau^{4}). 
        \end{align*}
        Define 
        $
        \mathcal{N}_{n}(E^{n}) = \mathcal{N}_{n,4}(E^{n}) + \tau^{2} \sum_{i=2}^{3}b_{i}(-\tau L)\left( Q_{n,i}(E^{n,i})\psi_{2,i}(-\tau L)g^{(1)}(t_{n}) \right),
        $
        and one finally obtains
        \begin{align}
            \begin{split}
                E^{n,4} =& \chi(-\tau L)E^{n} + \tau\mathcal{N}_{n}(E^{n}) - \tau^{3}\sum_{i=2}^{3}b_{i}(-\tau L)\left( J_{n}\psi_{2,i}(-\tau L)g^{(1)}(t_{n}) \right) \\
                & -\sum_{i=1}^{3}\tau^{i}\psi_{i}(-\tau L)g^{(i-1)}(t_{n}) + \mathcal{O}(\tau^{4}).
            \end{split}\label{proof_En4_end}
        \end{align}
        Note that for $b_{i}(-\tau L)$ and $\psi_{3}(-\tau L)$, there always exist bounded operators $\widetilde{b}_{i}$ and $\widetilde{\psi}_{3}$ such that:
        \begin{align*}
            b_{i}(-\tau L) = b_{i}(0) + \tau L \widetilde{b}_{i}(-\tau L), \quad \psi_{3}(-\tau L) = \psi_{3}(0) + \tau L \widetilde{\psi}_{3}(-\tau L).
        \end{align*}
       Therefore, under the order conditions in Table~\ref{tab. order condition} (with $s=3$), by imposing higher regularity requirements on $g$ (Assumption~\ref{ass. phi}), (\ref{proof_En4_end}) can be further simplified to
        \begin{align}
            \begin{split}
                E^{n,s+1} = \chi(-\tau L)E^{n} + \tau\mathcal{N}_{n}(E^{n}) + \mathcal{O}(\tau^{4}).
            \end{split} 
        \end{align}
        which immediately imply that
        \begin{equation}\label{est_RK_error}
            \|E^{n,s+1}\| \le (1+C\tau)\|E^n\|+C\,\tau^{s+1}.
        \end{equation}
    
        \medskip
        
        \textbf{Step 2. Correction Step.} 
        From the correction step given in \eqref{scm. Corrector 1 simplified}, it holds that
        $$
        E^{n+1} - \lambda^{n+1}p'(\phi^{n+1}) = E^{n,s+1}.
        $$
        Taking the $L^2$-norm squared on both sides of above, we obtain
        $$
        \|E^{n+1}\|^2 - 2\,\Bigl\langle \lambda^{n+1}p'(\phi^{n+1}),\,\phi^{n+1}-\phi(t_{n+1}) \Bigr\rangle + \|\lambda^{n+1}p'(\phi^{n+1})\|^2 = \|E^{n,s+1}\|^2.
        $$
        A key observation is that
        $$
        \Bigl\langle \lambda^{n+1}p'(\phi^{n+1}),\,\phi^{n+1}-\phi(t_{n+1}) \Bigr\rangle = -\Bigl\langle \lambda^{n+1},\,p'(\phi^{n+1})(\phi(t_{n+1})-\phi^{n+1}) \Bigr\rangle,
        $$
        and by employing a second-order Taylor expansion of $p(\cdot)$:
        $$
        p(\phi(t_{n+1})) - p(\phi^{n+1}) = p'(\phi^{n+1})(\phi(t_{n+1})- \phi^{n+1}) + \frac{1}{2}p''(\xi)\Bigl(\phi(t_{n+1})-\phi^{n+1}\Bigr)^2,
        $$
        where $\xi$ lies between $\phi^{n+1}$ and $\phi(t_{n+1})$, and noting that $p''(\xi)=-2$, we deduce
        \begin{align*}
            &\Bigl\langle \lambda^{n+1}p'(\phi^{n+1}),\, \phi^{n+1}-\phi(t_{n+1}) \Bigr\rangle  \\
            =& -\left\langle \lambda^{n+1}, p(\phi(t_{n+1})) \right\rangle + \left\langle \lambda^{n+1}, p(\phi^{n+1}) \right\rangle - \left\langle \lambda^{n+1}, \bigl( \phi(t_{n+1})-\phi^{n+1} \bigr)^2 \right\rangle \leq 0.
        \end{align*}
        Therefore, one immediately has
        $
        \|E^{n+1}\| \le \|E^{n,s+1}\|.
        $
        Combining this inequality with \eqref{est_RK_error} yields
        $$
        \|E^{n+1}\| \le (1+C\tau)\|E^n\|+C\,\tau^{s+1}.
        $$
        An application of the discrete Gronwall inequality then completes the proof.
    \end{proof}

    \begin{remark}
    For more complex phase-field models (e.g., the Cahn-Hilliard equation) and/or higher-order ETDRK satisfying the energy dissipation condition (if can be constructed in future), as long as the $L^2$ convergence holds on the plain ETDRK scheme, our analysis for the PC framework remains valid. In particular, the error accumulation introduced by the corrector remains unchanged.

    \end{remark}

    
    
    
    \section{Spatial discretization} \label{sec. 4}
    
    In this section, we discuss the effect of spatial discretizations on the preservation of the structure in PCC. 
    
    The predictor in practice can be performed using any spatial discretization technique in the user's favor, e.g., finite difference, finite element or spectral method. 
    Note that both the energy corrector (\ref{scm. Corrector 1}) and the maximum bound corrector (\ref{scm. Corrector 2}) after spatial discretization can still provide energy stability and MBP in the discrete sense. The key point lies in how the spatially discretized energy $\mathcal{E}_h\approx\mathcal{E}$  is defined for (\ref{scm. Corrector 1}) and whether (\ref{scm. Corrector 2}) can maintain the stability established  via (\ref{scm. Corrector 1}).
    
    More precisely, we denote by $\phi^{n+1}_{*,h},\phi^{n+1}_h$ the numerical solutions from spatially discretized \eqref{scm. Corrector 1} and \eqref{scm. Corrector 2}, respectively. Then, (\ref{scm. Corrector 1}) offers 
    $\mathcal{E}_h(\phi^{n+1}_{*,h})\leq \mathcal{E}_h(\phi^{n}_{h})$ with the chosen $\mathcal{E}_h$, and the outcome of (\ref{scm. Corrector 2}) can still be given by
    $\Phi(\phi^{n+1}_{*,h})$ with $\Phi$ defined in (\ref{Phi def}).
    We would like to discuss if 
    $\mathcal{E}_h(\phi^{n+1}_{h})\leq \mathcal{E}_h(\phi^{n+1}_{*,h})$ holds. 
    
    \subsection{Fourier spectral method} Suppose that we are now working with
    the Fourier spectral method \cite{shen2011spectral}, which is popular for phase-field models. Particularly, it makes ETD very efficient. 
    For simplicity, we assume the domain to be $\Omega = [0, 2\pi]^{d}$ with periodic boundary conditions for (\ref{eq. gradient flow}) and give some basic related notations. Let $M$ be an even positive integer with $h = \frac{2\pi}{M}$ the mesh size. 
    Denote
    $\mathcal{L}_{h}=\{\mathbf{l}\in\mathbb{Z}^{d}\ |\ -M/2\leq l_{k}\leq M/2-1,\; k = 1,\ldots, d \},$
    $\mathcal{J}_{h}=\{\mathbf{j}\in\mathbb{Z}^{d}\ |\ 0 \leq j_{k}\leq M-1,\; k = 1,\ldots, d \},$
    and $\Omega_{h}=\{\mathbf{x}_{\mathbf{j}}=h\mathbf{j}\ |\ \mathbf{j}\in \mathcal{J}_{h}\}$ the set of grid points on $\Omega$. 
    The projection $\mathcal{P}_{h}$  and the interpolation $\mathcal{I}_{h}$ operators are denoted as 
    \begin{align*}
    \bigl(\mathcal{P}_{h}u\bigr)(\mathbf{x}) = \sum_{\mathbf{l}\in\mathcal{L}_{h}} \widehat{u}_{\mathbf{l}} \mathrm{e}^{\text{i}\mathbf{l}\cdot \mathbf{x}},\quad 
    \bigl(\mathcal{I}_{h}u\bigr)(\mathbf{x}) = \sum_{\mathbf{l}\in\mathcal{L}_{h}} \widetilde{u}_{\mathbf{l}} \mathrm{e}^{\text{i}\mathbf{l}\cdot \mathbf{x}},       
    \end{align*}
    with $
        \widehat{u}_{\mathbf{l}} = \frac{1}{|\Omega|}\int_{\Omega} u(\mathbf{x}) \mathrm{e}^{-\text{i}\mathbf{l}\cdot\mathbf{x}} d\mathbf{x}
    $  the Fourier coefficient and $
        \widetilde{u}_{\mathbf{l}} = \frac{1}{M^{d}}\sum_{\mathbf{x}_{\mathbf{j}}\in\Omega_h} u(\mathbf{x}_{\mathbf{j}}) \mathrm{e}^{-\text{i}\mathbf{l}\cdot\mathbf{x}_{\mathbf{j}}} 
    $  the discrete Fourier coefficient.
    
    Under the Fourier spectral method discretization, it is convenient in practice to define $\phi^{n+1}_h=\mathcal{I}_{h}\Phi(\phi_{*,h}^{n+1})$ for PCC, and the discrete energy is indeed computed as 
    \begin{align}\label{energy I}
        \mathcal{E}_{h}[\phi] = \frac{1}{2} \left\| \nabla \mathcal{I}_{h}\phi \right\|^2 + \int_{\Omega} \mathcal{I}_{h}\bigl( F(\phi) \bigr)\;\mathrm{d}\mathbf{x}.
    \end{align}
    If so, on the one hand, we can see under Assumption~\ref{ass. f boundary} that
        \begin{align}
            \int_{\Omega}{ \mathcal{I}_{h}\bigl[ F(\phi_{h}^{n+1}) \bigr] }\;\mathrm{d}\mathbf{x} =& h^{d}\sum_{\mathbf{j}\in\mathcal{J}_{h}}F\bigl(\phi_{h}^{n+1}(\mathbf{x}_{\mathbf{j}})\bigr) \label{fml. F diminish FP} \\
            \leq& h^{d}\sum_{\mathbf{j}\in\mathcal{J}_{h}}F\bigl(\phi_{*,h}^{n+1}(\mathbf{x}_{\mathbf{j}})\bigr) = \int_{\Omega}{ \mathcal{I}_{h}\bigl[ F(\phi_{*,h}^{n+1}) \bigr] }\;\mathrm{d}\mathbf{x}. \nonumber
        \end{align}
        For the other part, we may get an extra interpolation error 
          \begin{align*}
            \bigl\| \nabla\phi_{h}^{n+1} \bigr\| = \bigl\| \nabla\mathcal{I}_{h}\bigl( \Phi\bigl( \phi_{*,h}^{n+1} \bigr) \bigr) \bigr\| \leq& \bigl\| \nabla\Phi\bigl( \phi_{*,h}^{n+1} \bigr) \bigr\| + \bigl\| \nabla\bigl( \mathcal{I}_{h} - I \bigr)\bigl( \Phi\bigl( \phi_{*,h}^{n+1} \bigr) \bigr) \bigr\| \\
            \leq& \bigl\| \nabla\phi_{*,h}^{n+1} \bigr\| + \bigl\| \nabla\bigl( \mathcal{I}_{h} - I \bigr)\bigl( \Phi\bigl( \phi_{*,h}^{n+1} \bigr) \bigr) \bigr\|,
        \end{align*}
        which might break $\mathcal{E}_h(\phi^{n+1}_{*,h})\leq \mathcal{E}_h(\phi^{n}_{h})$,  
        though this is never felt in our numerical experience. For theoretical guarantee, we consider the following. 
    \begin{proposition}
        \label{cor. energy diminish FP}
    Define  for PCC $\phi^{n+1}_h=\mathcal{P}_{h}\Phi(\phi_{*,h}^{n+1})$  and 
    $\mathcal{E}_{h}[\phi] = \frac{1}{2} \left\| \nabla \mathcal{P}_{h}\phi \right\|^2 + \int_{\Omega} \mathcal{I}_{h}\bigl( F(\phi) \bigr)\;\mathrm{d}\mathbf{x}$, then under Assumption~\ref{ass. f boundary},
    we have 
       $
            \mathcal{E}_{h}[\phi_{h}^{n+1}] \leq \mathcal{E}_{h}[\phi_{*,h}^{n+1}].
        $
    \end{proposition}
    \begin{proof}
        Clearly, we have
        $
            \bigl\| \nabla\phi_{h}^{n+1} \bigr\| = \bigl\| \nabla\mathcal{P}_{h}\bigl( \Phi\bigl( \phi_{*,h}^{n+1} \bigr) \bigr) \bigr\| =\bigl\| \mathcal{P}_{h}\nabla\Phi\bigl( \phi_{*,h}^{n+1} \bigr) \bigr\|
            \leq \bigl\| \nabla\phi_{*,h}^{n+1} \bigr\|.
       $ Note that under the Fourier spectral discretization, it is natural to have $\phi_{*,h}^{n+1}=\mathcal{P}_{h}\phi_{*,h}^{n+1}$.
        Combining with \eqref{fml. F diminish FP} yields the assertion.
    \end{proof}


    \subsection{Finite difference method}
    Another way to maintain the energy stability is to consider the finite difference discretization (for the energy). We take the two-dimensional case as an example (i.e., $d=2$) and consider the square region $\Omega = [x_{L}, x_{R}]^2$ with zero boundary conditions. Let the mesh size be $h= (x_{R} - x_{L})/M$ for some positive integer $M$. Denote 
    $
        \nabla_{h}\phi_{j,k} = \left( \delta_{x}^{+}\phi_{j,k},\,\delta_{y}^{+}\phi_{j,k} \right)^{\top}, $ 
      where $\delta_{x}^{+}\phi_{j,k} = \frac{1}{h}( \phi_{j+1,k} - \phi_{j,k}),\ 
    \delta_{y}^{+}\phi_{j,k} = \frac{1}{h}( \phi_{j,k+1} - \phi_{j,k})$.

    \begin{corollary}\label{cor. energy diminish FD}
    Define the discrete energy for PCC as 
    \begin{align*}
    \mathcal{E}_{h}[\phi] =  h^2\sum_{j=1}^{M-1}\sum_{k=1}^{M-1}\left(\frac12|\nabla_h\phi_{j,k}|^2+F(\phi_{j,k})\right). 
    \end{align*} Then under Assumption~\ref{ass. f boundary},  we have 
        $
            \mathcal{E}_{h}[\phi_{h}^{n+1}] \leq \mathcal{E}_{h}[\phi_{*,h}^{n+1}].
        $
    \end{corollary}
    \begin{proof}
        By definition of $\Phi$, we have $\phi_{j,k}^{n+1} = \Phi(\phi_{*,j,k}^{n+1})$ and
        \begin{align*}
            \bigl| \delta_{x}^{+}\phi_{j,k}^{n+1} \bigr| &
            = \frac{1}{h}\Bigl| \Phi(\phi_{*,j+1,k}^{n+1}) - \Phi(\phi_{*,j,k}^{n+1}) \Bigr| 
            \leq \frac{1}{h}\Bigl| \phi_{*,j+1,k}^{n+1} - \phi_{*,j,k}^{n+1} \Bigr|
            = \bigl| \delta_{x}^{+}\phi_{*,j,k}^{n+1} \bigr|.
        \end{align*}
        Similarly,  $\bigl| \delta_{y}^{+}\phi_{j,k}^{n+1} \bigr| \leq \bigl| \delta_{y}^{+}\phi_{*,j,k}^{n+1} \bigr|$ and so 
        $
            | \nabla_{h}\phi_{j,k}^{n+1} |^2 \leq | \nabla_{h}\phi_{*,j,k}^{n+1} |^2. 
        $
        The argument for the nonlinear part goes as  \eqref{fml. F diminish FP}, which leads to the assertion.
    \end{proof}
    
    
    \section{Numerical experiment} \label{sec. 5}
    
    In this section, we carry out numerical experiments to demonstrate the effectiveness of the proposed framework: PCC and PC. In all of our examples, to make computations efficient, we impose periodic boundary conditions for \eqref{eq. gradient flow} and utilize the Fourier spectral method  for spatial discretizations. The energy will be calculated as \eqref{energy I}.
    
    \subsection{Setup of the predictor} \label{sec. 5.1} 
    It is needed to clarify first the used predictors and their short names for later reference. 
    The predictors for the tests throughout the section are the ETDRK class of schemes, including the energy stable ones and unstable ones. The energy stable ones that satisfy Assumption~\ref{ass. A} are the following three ($s=1,2,3$) from  \cite{fu2024higher}.
    {\small
    $$
    \begin{array}{lc}
        \text{ETDRK1}: & 
        \begin{array}{c|c}
            0 & 0 \\
            \hline
            & \varphi_{1}(z)
        \end{array} \\[2mm]
        \text{ETDRK2}: & 
        \begin{array}{c|cc}
            0 &   &   \\
            1 & \varphi_{1}(z) &   \\
            \hline
            & \varphi_{1}(z)-\varphi_{2}(z) & \varphi_{2}(z)
        \end{array} \\[2mm]
        \text{ETDRK3}: & 
        \begin{array}{c|ccc}
            0 &   &   &   \\
            1 & \varphi_{1}(z) &    &   \\
            \frac{2}{3} & \frac{2}{3}\varphi_{1}\bigl(\frac{2}{3}z\bigr)-\frac{4}{9}\varphi_{2}\bigl(\frac{2}{3}z\bigr) & \frac{4}{9}\varphi_{2}\bigl(\frac{2}{3}z\bigr) &     \\
            \hline
            & \frac{3}{4}\varphi_{1}(z)-\varphi_{2}(z) & \varphi_{2}(z)-\frac{1}{2}\varphi_{1}(z)  & \frac{3}{4}\varphi_{1}(z)
        \end{array}
    \end{array}
    $$}
    The energy unstable ones  which we shall refer as U-ETDRK, are the following two ($s=3,4$).
    {\small
    $$
    \begin{array}{lc}
        \text{U-ETDRK3}:\\ 
        \begin{array}{c|ccc}
            0 &   &   &   \\
            \frac{1}{2} & \frac{1}{2}\varphi_{1}\bigl(\frac{1}{2}z\bigr) &    &   \\
            1 & -\varphi_{1}(z) &  2\varphi_{1}(z)   &     \\
            \hline
            & 4\varphi_{3}(z)-3\varphi_{2}(z)+\varphi_{1}(z) & -8\varphi_{3}(z)+4\varphi_{2}(z)  & 4\varphi_{3}(z)-\varphi_{2}(z)
        \end{array} \\
        \\[2mm]
        \text{U-ETDRK4}: \\ 
        \begin{array}{c|cccc}
            0 &   &   &    &   \\
            \frac{1}{2} & \frac{1}{2}\varphi_{1}\bigl(\frac{1}{2}z\bigr) &    &    &   \\
            \frac{1}{2} & \frac{1}{2}\varphi_{1}\bigl(\frac{1}{2}z\bigr)-\varphi_{2}\bigl(\frac{1}{2}z\bigr) &  \varphi_{2}\bigl(\frac{1}{2}z\bigr)  &    &   \\
            1 & \varphi_{1}(z)-2\varphi_{2}(z) &  0   &  2\varphi_{2}(z)   &     \\
            \hline
            & \varphi_{1}(z)-3\varphi_{2}(z)+4\varphi_{3}(z) & 2\varphi_{2}(z)-4\varphi_{3}(z)  & 2\varphi_{2}(z)-4\varphi_{3}(z)  & 4\varphi_{3}(z)-\varphi_{2}(z)
        \end{array}
    \end{array}
    $$}
    
    The above ETDRK schemes may be used as plain methods, i.e., without correction, for comparison purposes.  When they are used as predictors in our PCC or PC framework, we will add a suffix to their names. For example, if U-ETDRK3 is used for PCC (\ref{scm. Predictor}), we will refer to the whole scheme as U-ETDRK3-PCC, and if ETDRK3 is used for PC (\ref{scm. ETDRK Predictor simplified}), it will be referred as ETDRK3-PC.
    
    \subsection{Allen-Cahn equation}
    Our first numerical experiment 
    considers an Allen-Cahn equation, i.e.,  $\mathcal{G} = I$ in (\ref{eq. gradient flow}), with the double-well type potential
    \begin{align*}
        F(\phi) = \frac{1}{4\varepsilon^2}(\phi^{2}-\beta)^{2},
    \end{align*}
    which is the most classical case in applications. Here, $\varepsilon,\beta>0$ will be given parameters. It is well known that the model satisfies MBP: When the initial data of (\ref{eq. gradient flow}) is within the interval $[-\beta, \beta]$, the solution will remain in $[-\beta, \beta]$ for all time.
    
    \begin{example}\label{Ex 1}
        Consider (\ref{eq. gradient flow}) in two dimensions, i.e., $d=2, \mathbf{x}=(x,y)$, with the initial data 
         \begin{align*}
            \phi(x,y,0) = \tanh\Big( \frac{1-\sqrt{(x-\pi)^2+(y-\pi)^2}}{\sqrt{2}\varepsilon} \Big).
        \end{align*}
        Then $\beta=1$ is fixed.
        Set the computational domain as $\Omega=[0,2\pi]^{2}$ with a $256\times 256$ uniform spatial grid. The stabilizing parameter of ETD is chosen as $S = 1/\varepsilon^2$.
    \end{example}

    \begin{table}[!ht]
        \centering
        \caption{Temporal accuracy test for the proposed schemes: $L^2$-errors under different $\tau=T/N$.}
        \resizebox{\textwidth}{!}{%
            \begin{tabular}{@{}cccccccccccccc@{}}
                \toprule
                \multirow{2}{*}{$N$} & \multicolumn{2}{c}{ETDRK1-PC} & \multicolumn{2}{c}{ETDRK2-PC} & \multicolumn{2}{c}{ETDRK3-PC} & \multicolumn{2}{c}{U-ETDRK3-PCC} & \multicolumn{2}{c}{U-ETDRK4-PCC} \\ 
                \cmidrule(lr){2-3} \cmidrule(lr){4-5} \cmidrule(lr){6-7} \cmidrule(lr){8-9} \cmidrule(lr){10-11}
                & $L^2$ error & rate & $L^2$ error & rate & $L^2$ error & rate & $L^2$ norm & rate & $L^2$ norm & rate \\ 
                \midrule
                50  & 1.05E-1 & -   & 1.52E-2 & -   & 2.26E-3 & -   & 1.08E-3 & -   & 3.57E-5 & - \\
                100 & 5.48E-2 & 0.94 & 4.26E-3 & 1.83 & 3.42E-4 & 2.72 & 1.58E-4 & 2.77 & 2.61E-6 & 3.77 \\
                200 & 2.80E-2 & 0.97 & 1.13E-3 & 1.91 & 4.74E-5 & 2.85 & 2.15E-5 & 2.88 & 1.77E-7 & 3.88 \\
                400 & 1.42E-2 & 0.98 & 2.91E-4 & 1.96 & 6.25E-6 & 2.92 & 2.80E-6 & 2.94 & 1.12E-8 & 3.97 \\
                800 & 7.13E-3 & 0.99 & 7.19E-5 & 2.01 & 8.03E-7 & 2.96 & 3.58E-7 & 2.97 & 4.30E-10 & 4.71 \\
                \bottomrule
            \end{tabular}%
        }
        \label{tab:accuracy_test}
    \end{table}
    
    \begin{figure}[!ht] 
        \centering
        \includegraphics[width=0.43\textwidth]{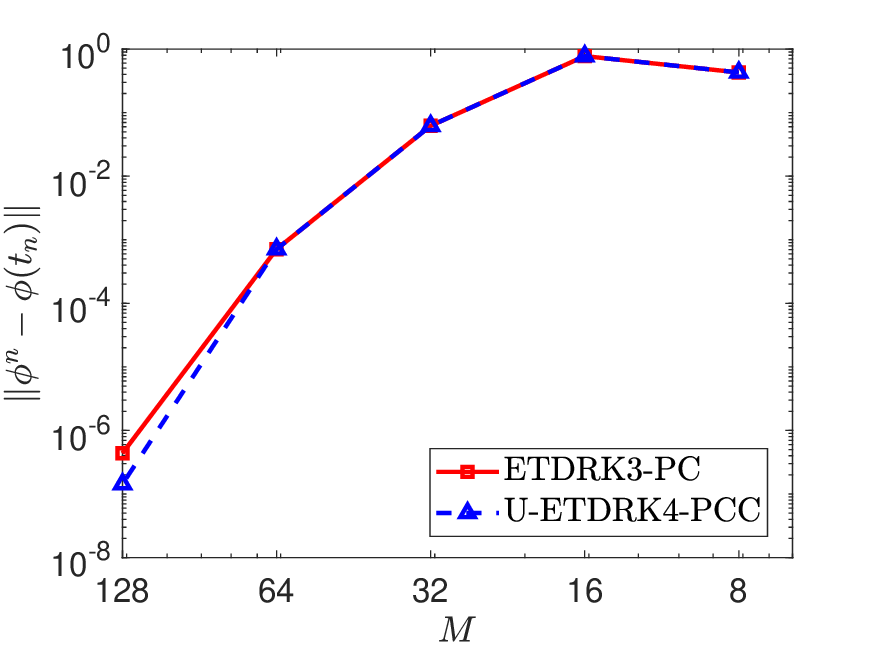} 
        \caption{(Example \ref{Ex 1}) 
         Spatial error of proposed schemes under  $M\times M$ uniform grid.}
        \label{fig:ex1_loglog_err}
    \end{figure}

    \begin{figure}[!ht] 
        \centering
        \includegraphics[width=0.24\textwidth, height=0.2\textwidth]{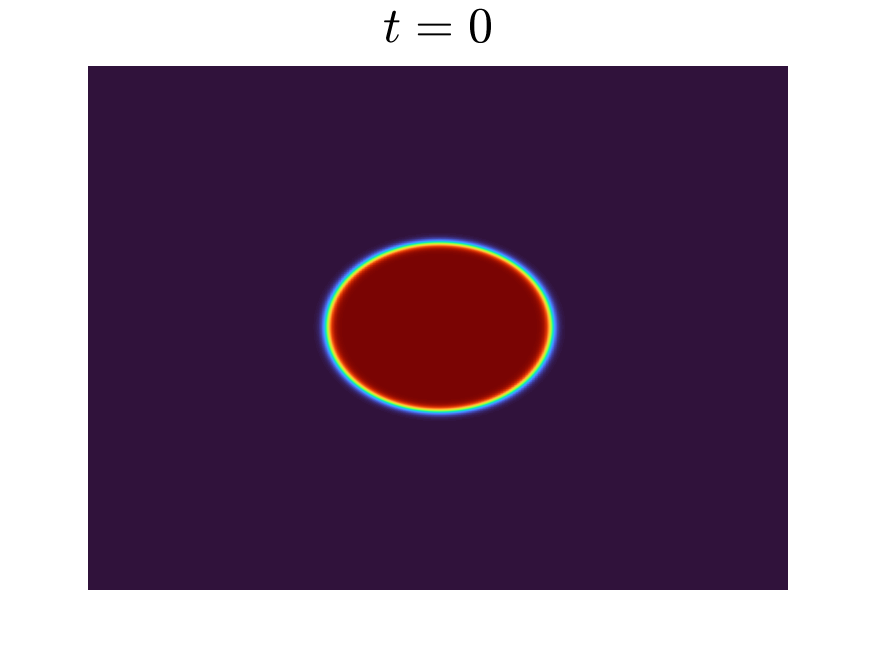} 
        \includegraphics[width=0.24\textwidth, height=0.2\textwidth]{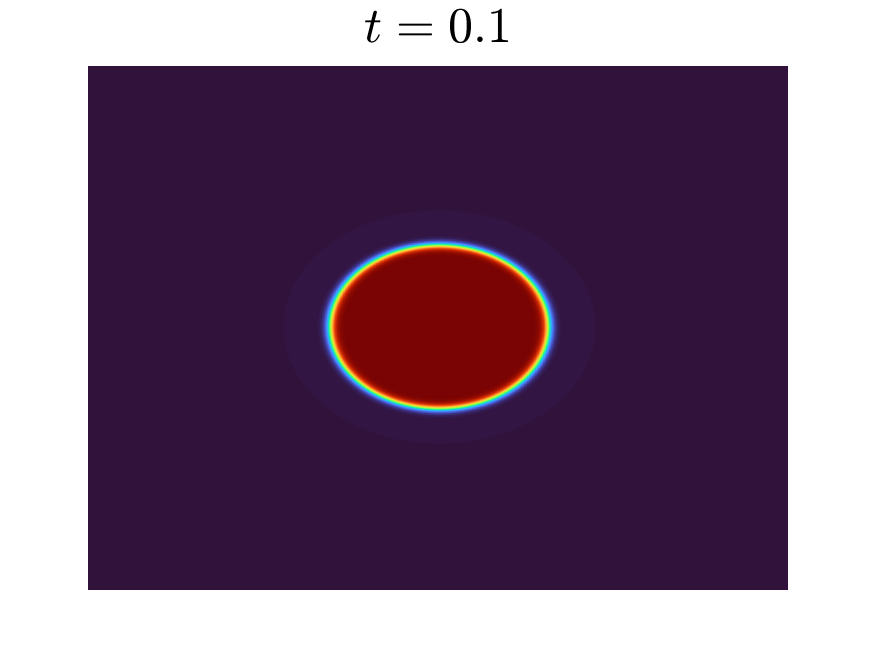} 
        \includegraphics[width=0.24\textwidth, height=0.2\textwidth]{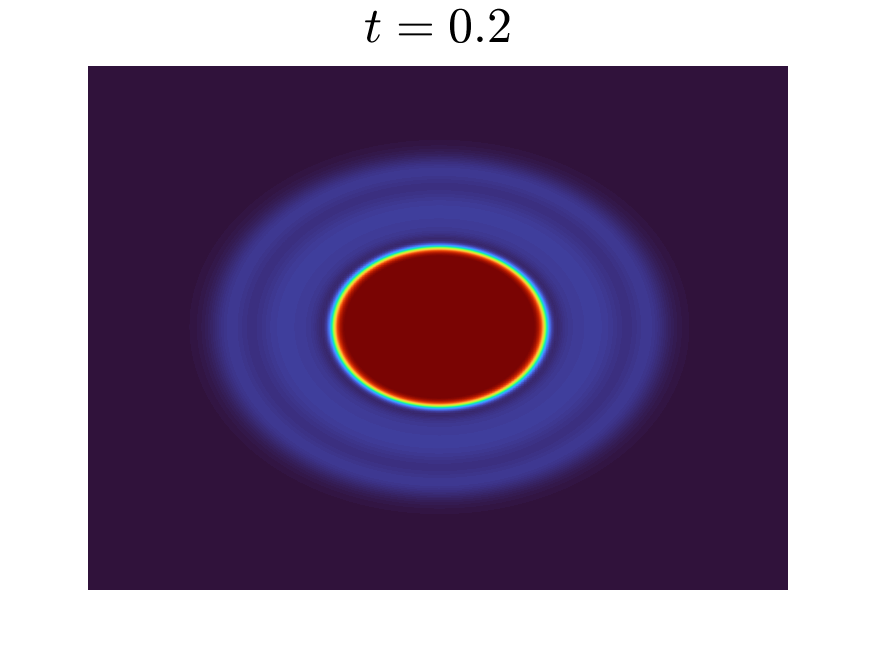} 
        \includegraphics[width=0.24\textwidth, height=0.2\textwidth]{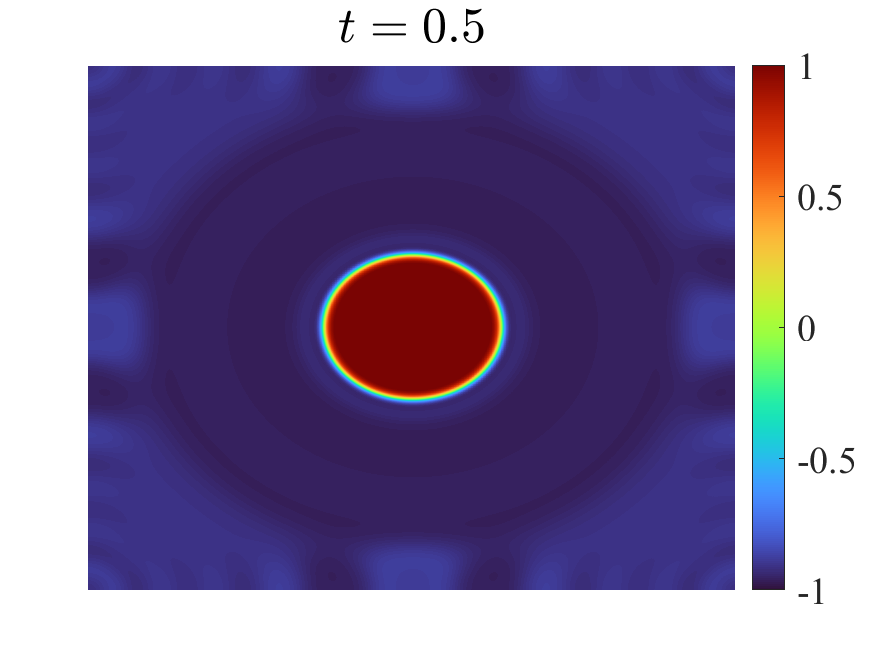}
        \\
        \includegraphics[width=0.24\textwidth, height=0.2\textwidth]{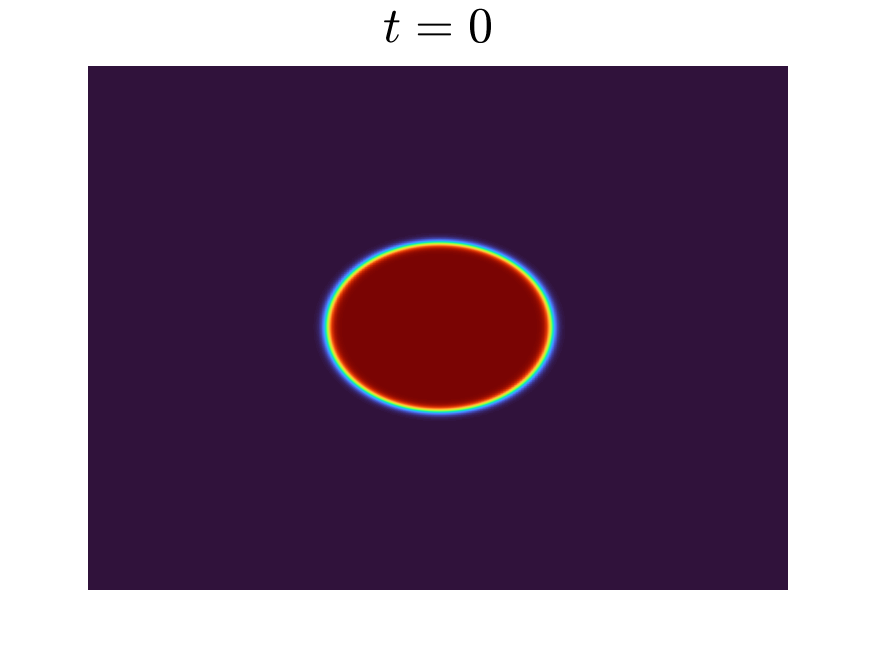} 
        \includegraphics[width=0.24\textwidth, height=0.2\textwidth]{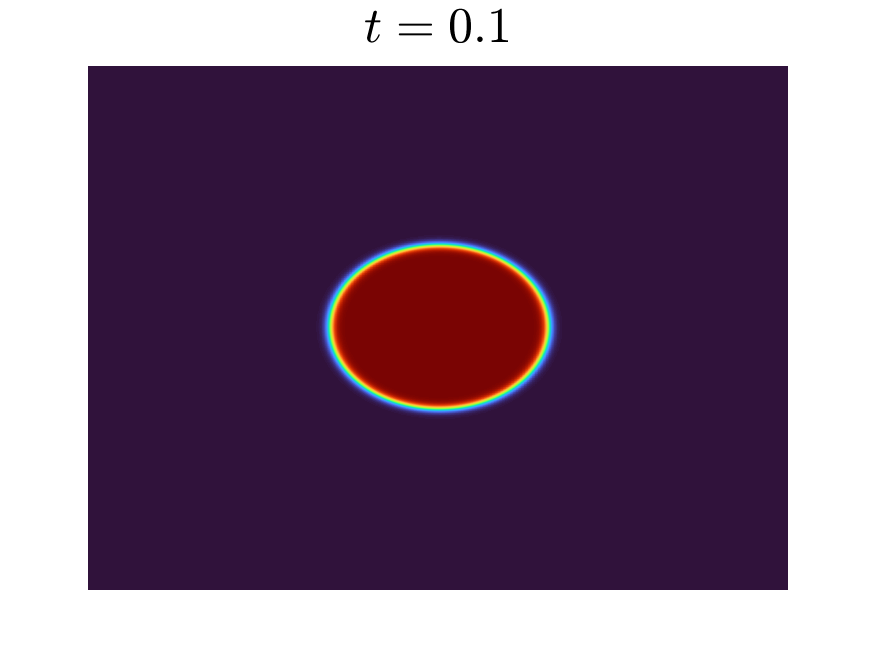} 
        \includegraphics[width=0.24\textwidth, height=0.2\textwidth]{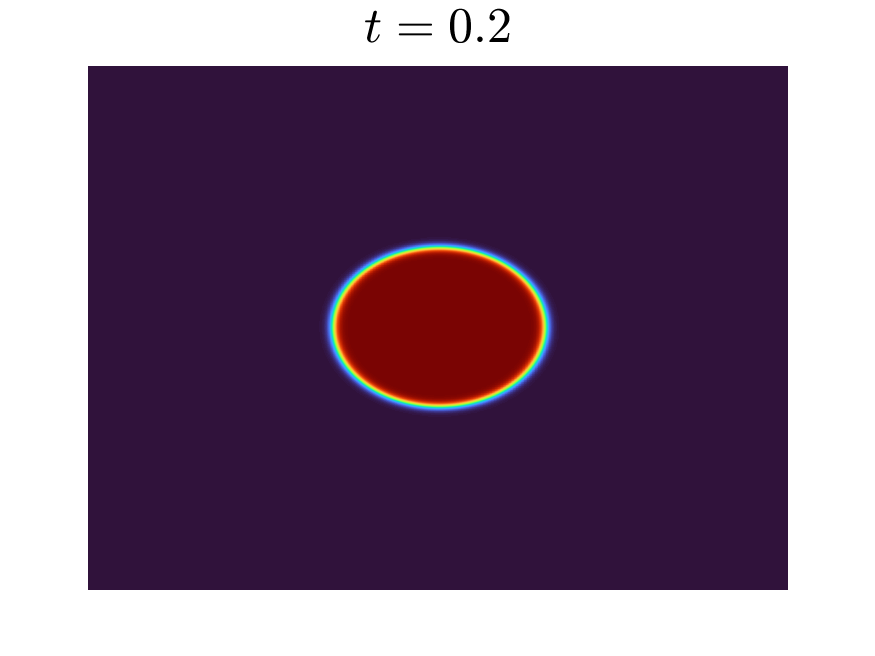} 
        \includegraphics[width=0.24\textwidth, height=0.2\textwidth]{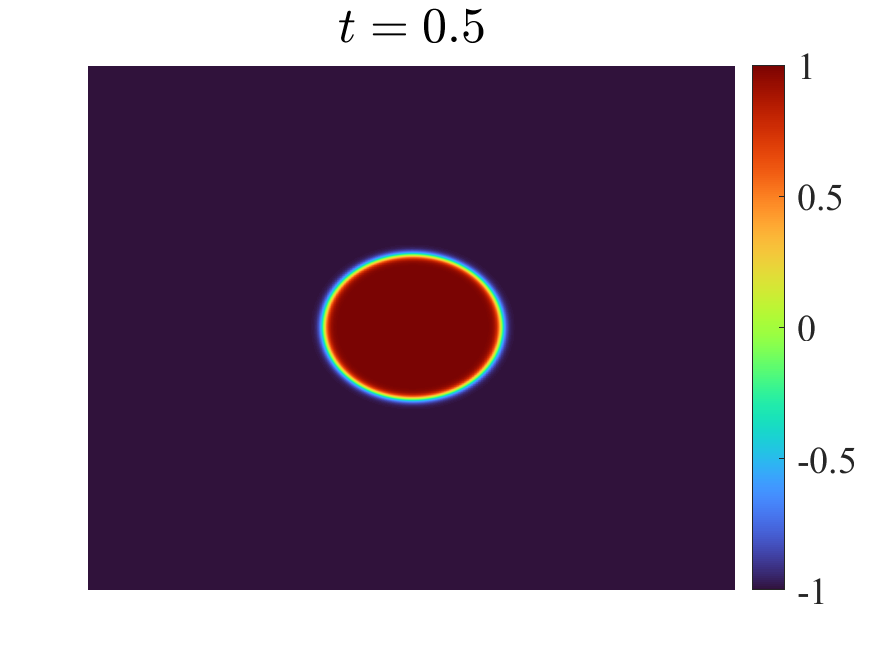} 
        \caption{(Example \ref{Ex 1}) Contour plots of the numerical solution $\phi^n$ at $t = 0$, $0.1$, $0.2$, $0.5$ with $\tau = 0.01$. 1st row: U-ETDRK3; 2nd row: U-ETDRK3-PCC.}
        \label{fig:ex1_snapshot s=3}
    \end{figure}
     \begin{figure}[h!] 
        \centering
        \subfloat[]{%
    \includegraphics[width=0.31\textwidth]{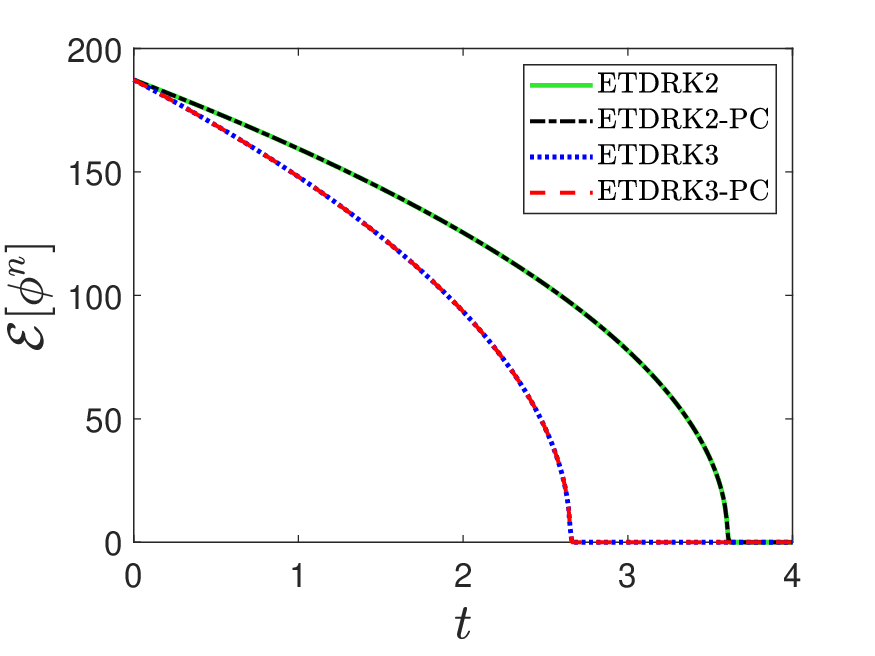} 
            \label{fig:energy_s2}
        }
        \subfloat[]{%
            \includegraphics[width=0.31\textwidth]{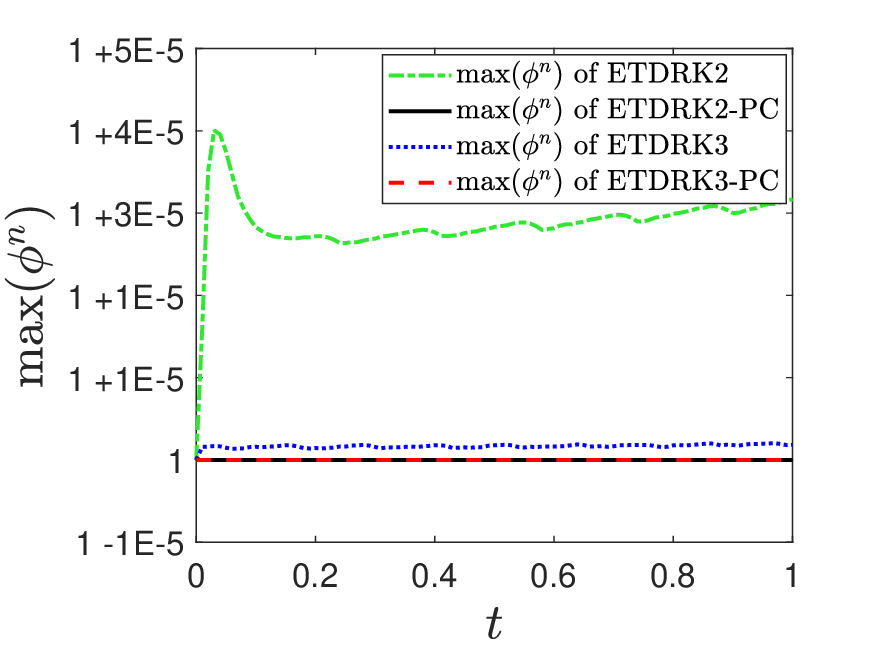} 
            \label{fig:boundary_s2}
        }
        \subfloat[]{%
            \includegraphics[width=0.31\textwidth]{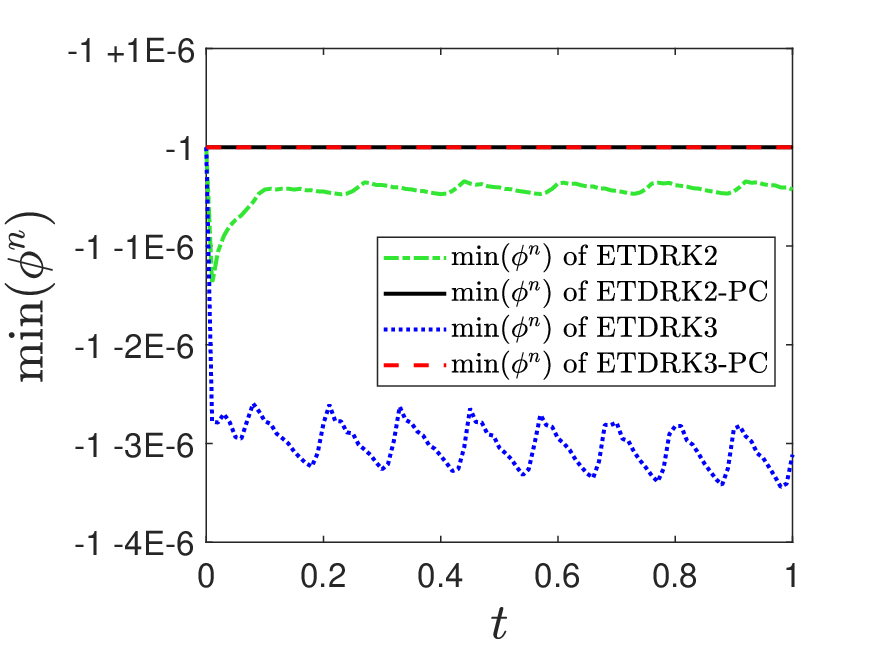} 
            \label{fig:Lboundary_s2}
        }
        \caption{(Example \ref{Ex 1}) Time evolutions of energy, upper and lower bounds of the numerical solution $\phi^{n}$ from ETDRK and ETDRK-PC schemes with  $\tau = 0.01$.}
        \label{fig:test add}
    \end{figure}
    
    \begin{figure}[h!] 
        \centering
        \includegraphics[width=0.32\textwidth]{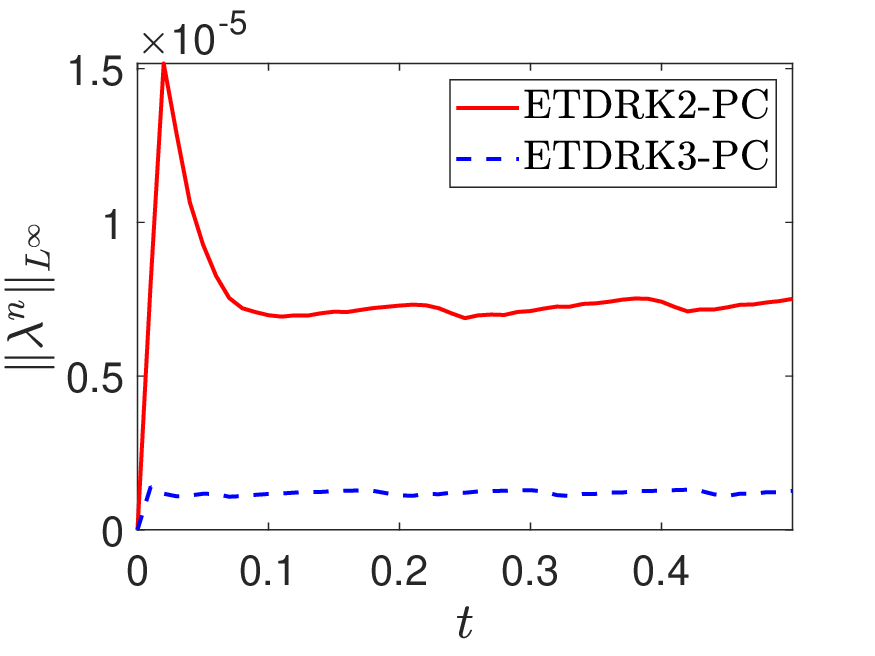}
        \includegraphics[width=0.32\textwidth]{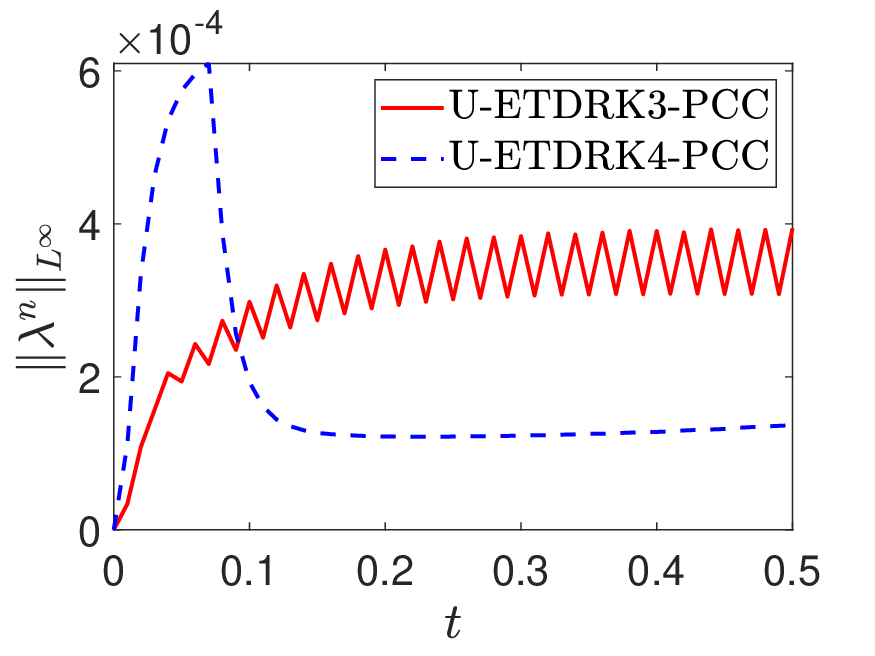}
        \includegraphics[width=0.32\textwidth]{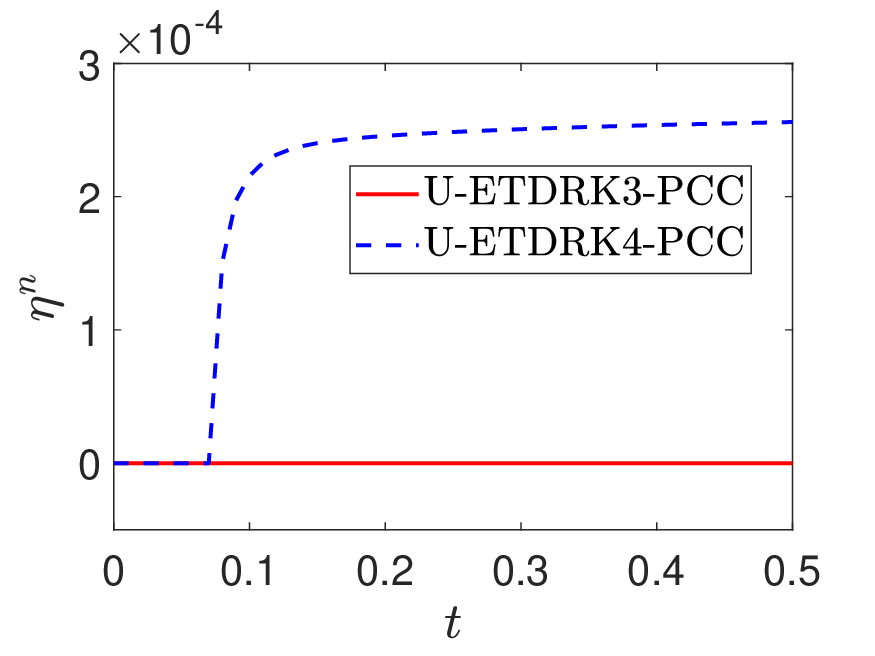} 
        \caption{(Example \ref{Ex 1}) Values of $\|\lambda^{n}\|_{L^\infty}$ and $\eta^{n}$ during the computations for different schemes with $\tau = 0.01$.}
        \label{fig:ex1_auxpara}
    \end{figure}
    
    We solve Example \ref{Ex 1} by the presented numerical methods. Firstly, we validate the optimal convergence order of the concerned schemes.  
    To do so, we fix $\varepsilon^2 = 0.01$ and measure the (absolute) error of the solution in the $L^{2}$-norm under different time step $\tau=T/N$ at $t = T = 0.1$, where the reference solution is computed by the U-ETDRK4-PCC scheme with $\tau=10^{-4}$. Table \ref{tab:accuracy_test} presents the temporal errors of PC and PCC schemes, where the spatial errors ($256\times256$ grid) are negligible in the setup. The numerical results show that all the schemes tested can achieve the optimal convergence rates as their plain versions. In particular, the theoretical result in Theorem~\ref{thm. convergency} is confirmed.     
    The spatial spectral accuracy of PCC and PC are observed in Fig.~\ref{fig:ex1_loglog_err}, where $N = 1000$ is fixed.

    Then, we test the structure-preserving effect of the proposed strategy. Here, $\varepsilon^{2}=0.001$ is fixed. 
    The behavior of the numerical energy and the extreme values of the solutions of U-ETDRK and U-ETDRK-PCC have already been given in Fig. \ref{fig:test}, which validates the energy stability and MBP resulting from PCC. 
    Fig. \ref{fig:ex1_snapshot s=3} further shows the contour plots of the numerical solution $\phi^n$ obtained from U-ETDRK3 or U-ETDRK3-PCC at different time. From these plots, it is evident that U-ETDRK3 exhibits significant spurious oscillations, whereas U-ETDRK3-PCC gives the desired simulation.     
    In addition, we illustrate the structure preservation of ETDRK-PC in Fig. \ref{fig:test add}. As shown in Fig.~\ref{fig:energy_s2}, the bound corrector of PC does not break the inherent energy stability of the plain scheme, and the energy decay trajectory of ETDRK-PC is nearly identical to that of ETDRK. Figs.~\ref{fig:boundary_s2}\&\ref{fig:Lboundary_s2} show that the energy stability of ETDRK does not implies MBP, and PC can fix the issue. At last, we show the values of the two Lagrange multipliers during the above computations in Fig. \ref{fig:ex1_auxpara}.

    \begin{example}\label{Ex 2}
        Consider (\ref{eq. gradient flow}) in three dimensions, i.e., $d=3, \mathbf{x}=(x,y,z)$, with the initial data 
        \begin{align}
            \phi(x,y,0) =& \tanh\Big( \frac{\frac{\pi}{6}-\sqrt{(x+\frac{\pi}{4})^2 + (y+\frac{\pi}{4})^2 + z^2 }}{\sqrt{2}\varepsilon} \Big) \nonumber\\
            & + \tanh\Big( \frac{\frac{\pi}{5}-\sqrt{(x+\frac{\pi}{4})^2 + (y-\frac{\pi}{4})^2 + z^2 }}{\sqrt{2}\varepsilon} \Big) \nonumber\\
            & + \tanh\Big( \frac{\frac{\pi}{6}-\sqrt{(x-\frac{\pi}{4})^2 + (y-\frac{\pi}{4})^2 + z^2 }}{\sqrt{2}\varepsilon} \Big) \nonumber\\
            & + \tanh\Big( \frac{\frac{\pi}{6}-\sqrt{x^2 + y^2 + (z-\frac{\pi}{3})^2 }}{\sqrt{2}\varepsilon} \Big)=:\phi_0(x,y)\label{phi0 def}
        \end{align}
        Then, $\beta=1$ is fixed.
        Set the computational domain as $\Omega=[-\pi,\pi]^{3}$ with a $128\times 128$ uniform spatial grid. Set $\varepsilon = 0.1$ and the stabilizing parameter of ETDRK is chosen as $S = 1/\varepsilon^2$.
    \end{example}

 \begin{figure}[!ht] 
       \centering
       \includegraphics[width=0.32\textwidth]{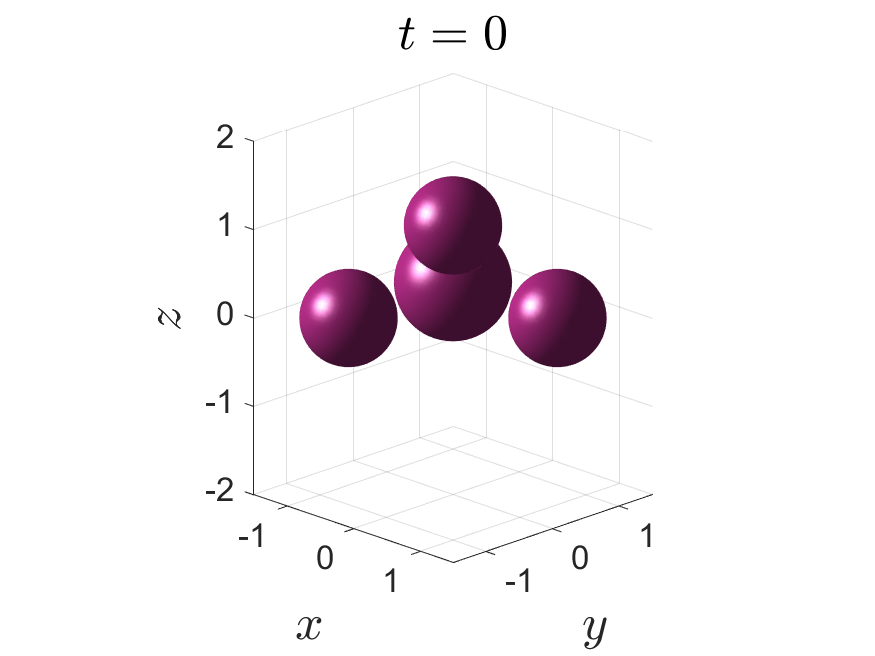} 
       \includegraphics[width=0.32\textwidth]{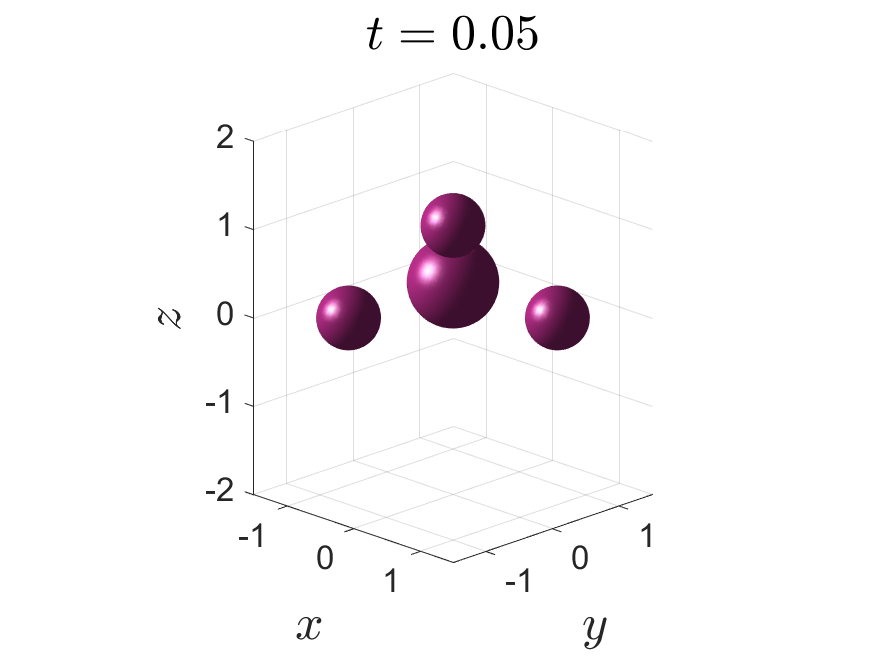} 
       \includegraphics[width=0.32\textwidth]{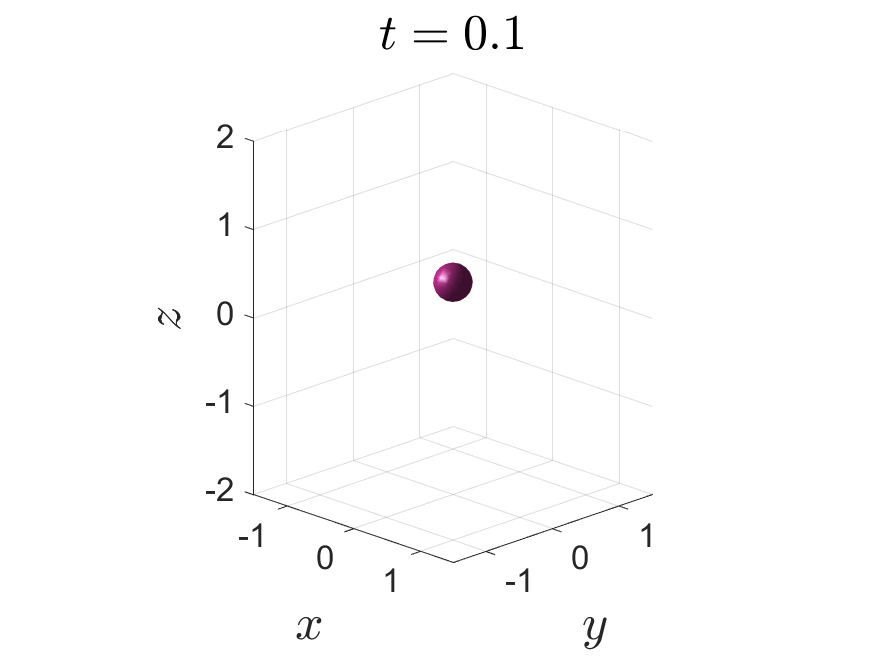}
       \caption{(Example \ref{Ex 2}) Snapshots of the numerical solution from U-ETDRK3-PCC: the iso-surfaces plots correspond to $\phi^{n} = 0$ at time $t= 0$, $0.05$ and $0.1$ under $\tau = 0.01$.} \label{fig:ex2_snapshot s=3}
   \end{figure}
    
    \begin{figure}[h!] 
      \centering
       \includegraphics[width=0.4\textwidth]{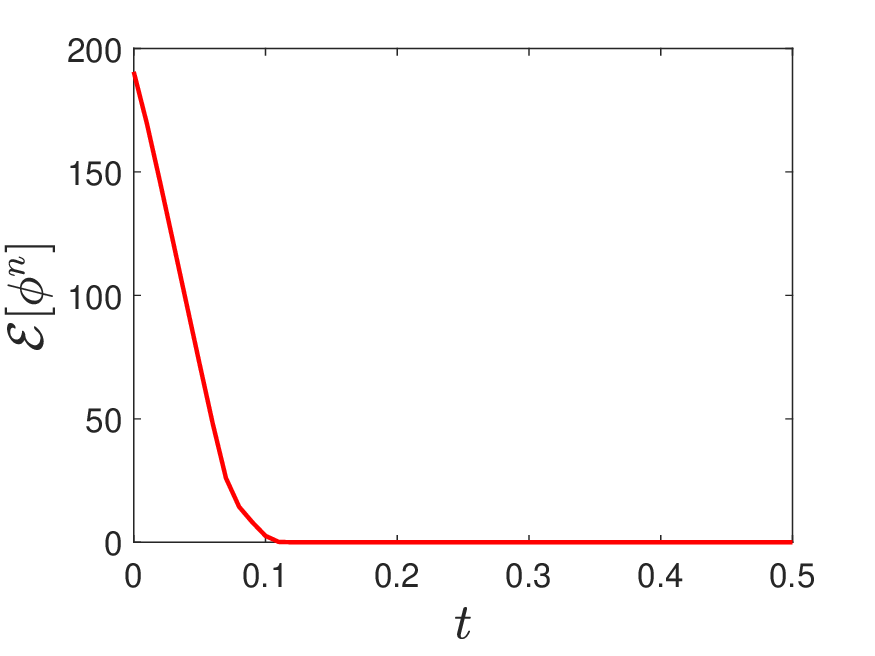}
       \includegraphics[width=0.4\textwidth]{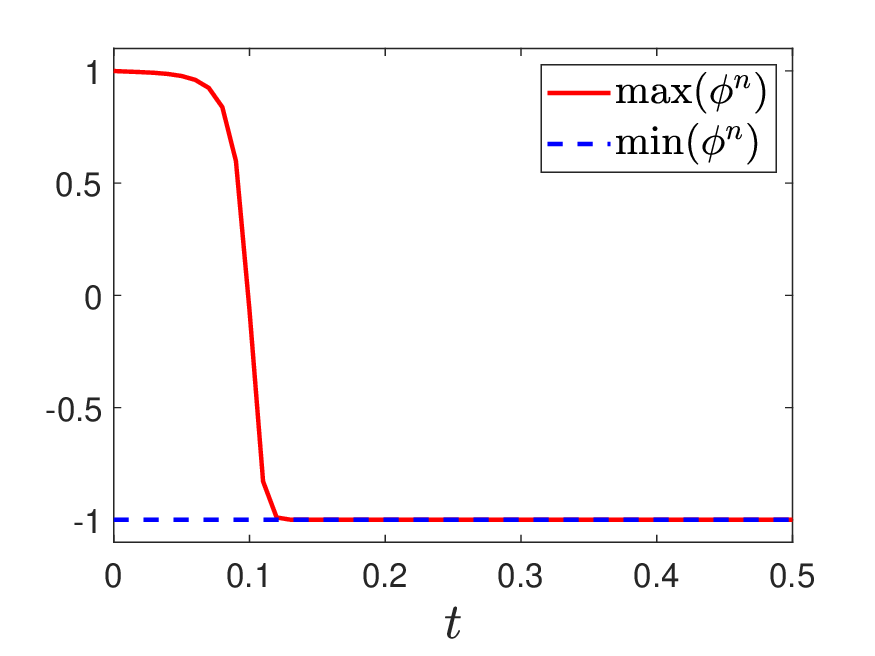}
       \caption{(Example \ref{Ex 2}) Time evolutions of energy, upper and lower bounds of the numerical solution $\phi^{n}$ from U-ETDRK3-PCC scheme with  $\tau = 0.01$.}
       \label{fig:ex2_stability}
  \end{figure}
    
    In Fig.~\ref{fig:ex2_snapshot s=3}, we display the iso-surfaces plot of the numerical solution computed by U-ETDRK3-PCC for Example~\ref{Ex 2} at various time points.  Fig.~\ref{fig:ex2_stability} demonstrates the structure preservation of our method. As illustrated, the PCC approach effectively maintains both MBP and energy stability in the three-dimensional case.

    %

    \subsection{Cahn-Hilliard equation}
    
    Next, we consider the Cahn-Hilliard equation, i.e., $\mathcal{G} = -\Delta$ in (\ref{eq. gradient flow}), with the Flory-Huggins potential function:
    \begin{align*}
        F(\phi) = \frac{1}{\varepsilon^2} \left[ (\beta + \phi) \operatorname*{ln}(\beta + \phi) + (\beta - \phi) \operatorname*{ln}(\beta - \phi) - \frac{\theta_{0}}{2} \phi^2 \right],
    \end{align*}
    where $\varepsilon,\,\beta$ and $\theta_{0}$ are  positive given parameters. If initially  $\left\|\phi(\cdot,0)\right\|_{L^{\infty}}$ is uniformly below $\beta$, then it is expected that the solution of the Cahn-Hilliard equation remains so in dynamics, i.e., $-\beta + \delta\leq \phi\leq\beta - \delta$ for some $\delta \in (0, \beta)$ \cite{Debussche1995CahnHilliard, Elliott1996CahnHilliard}.

    \begin{example} \label{Ex 3}
        Consider a two-dimensional case on the computational domain $\Omega=[0,2\pi]^{2}$ with a $128\times 128$ spatial grid. The  parameters are set as
        $\beta = 1$, $\varepsilon = 0.1$, $S = 1/\varepsilon^2$, $\theta_{0}=3$. 
        The initial data for the flow (\ref{eq. gradient flow}) is chosen as
        \begin{align*}
            \phi(x,y,0) = 0.2 + 0.05\operatorname*{Rand}(x,y),
        \end{align*}
        where  $\operatorname*{Rand}(x,y)$ denotes a random function with values uniformly distributed in $(-1,1)$.
    \end{example}

    \begin{figure}[h!] 
        \centering
        \includegraphics[width=0.32\textwidth]{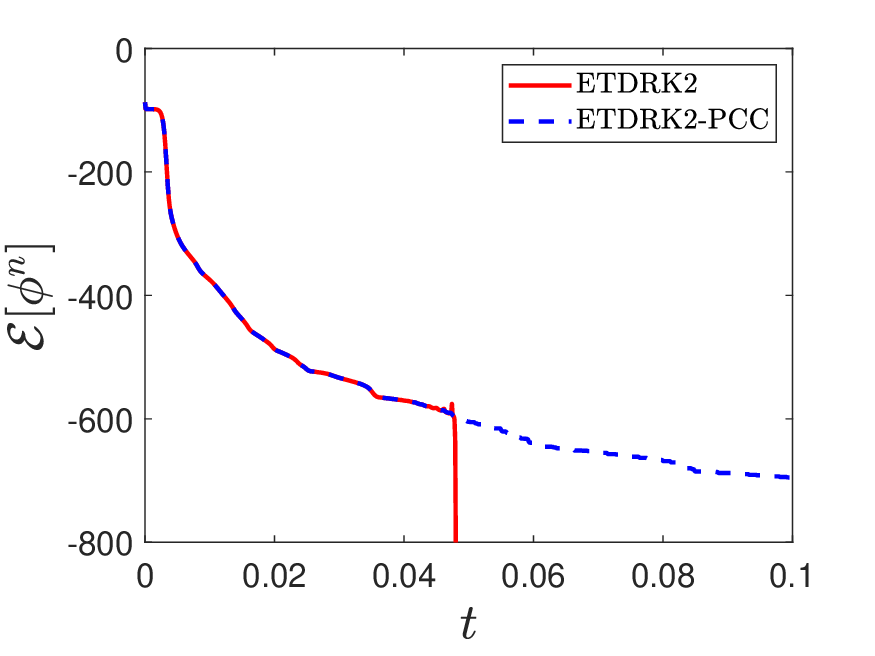}
        \includegraphics[width=0.32\textwidth]{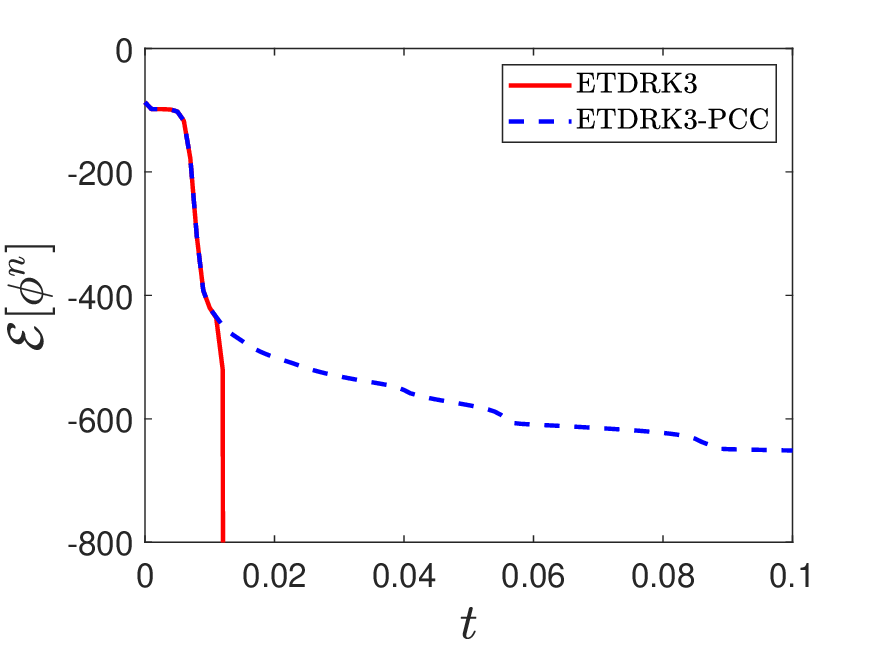} 
        \includegraphics[width=0.32\textwidth]{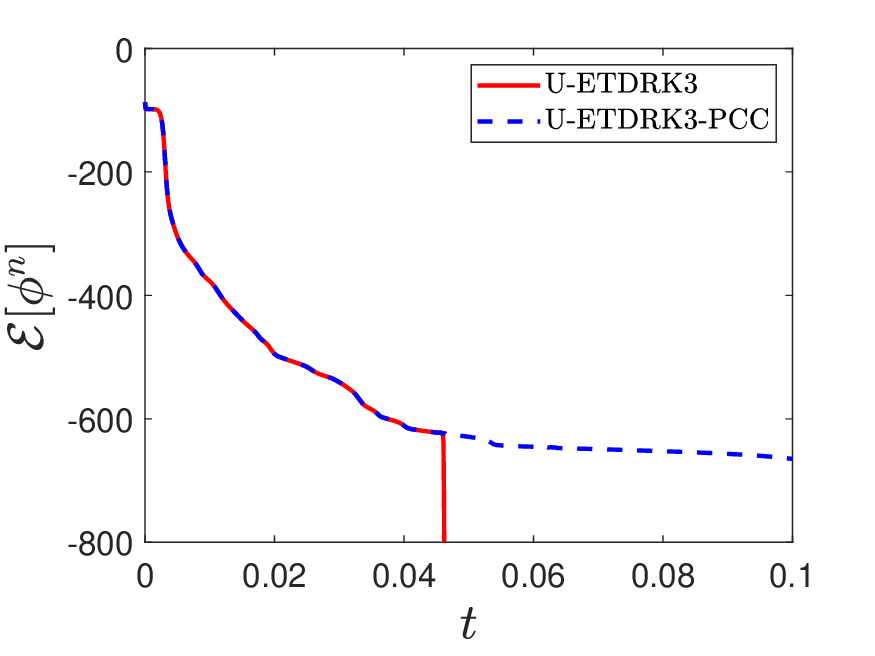}\\
        \includegraphics[width=0.32\textwidth]{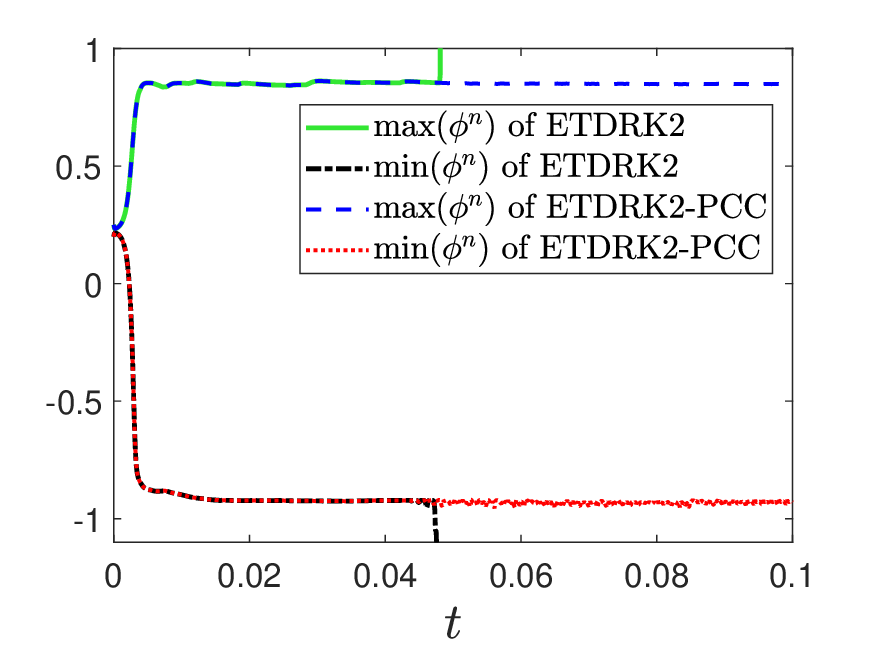}
        \includegraphics[width=0.32\textwidth]{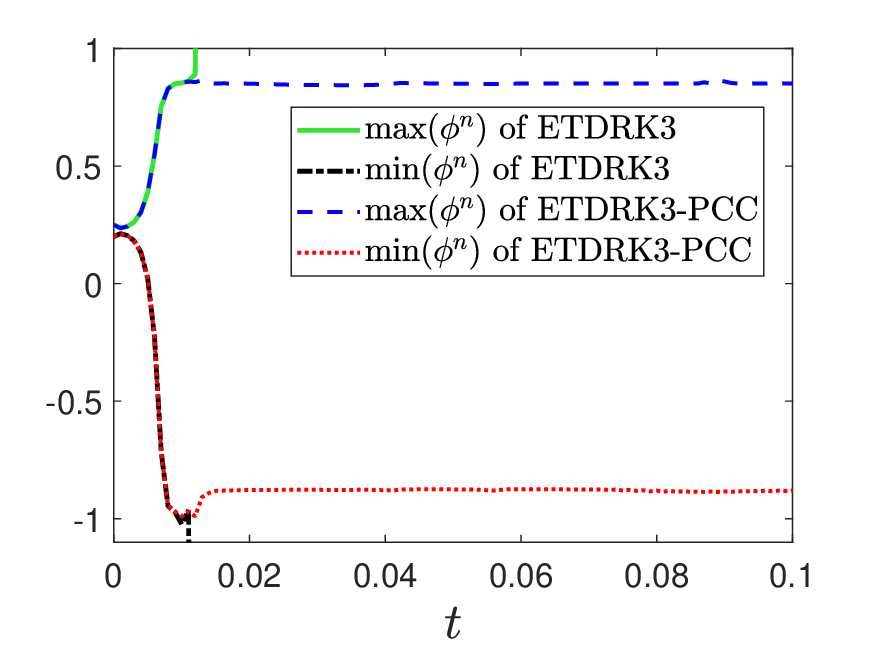} 
        \includegraphics[width=0.32\textwidth]{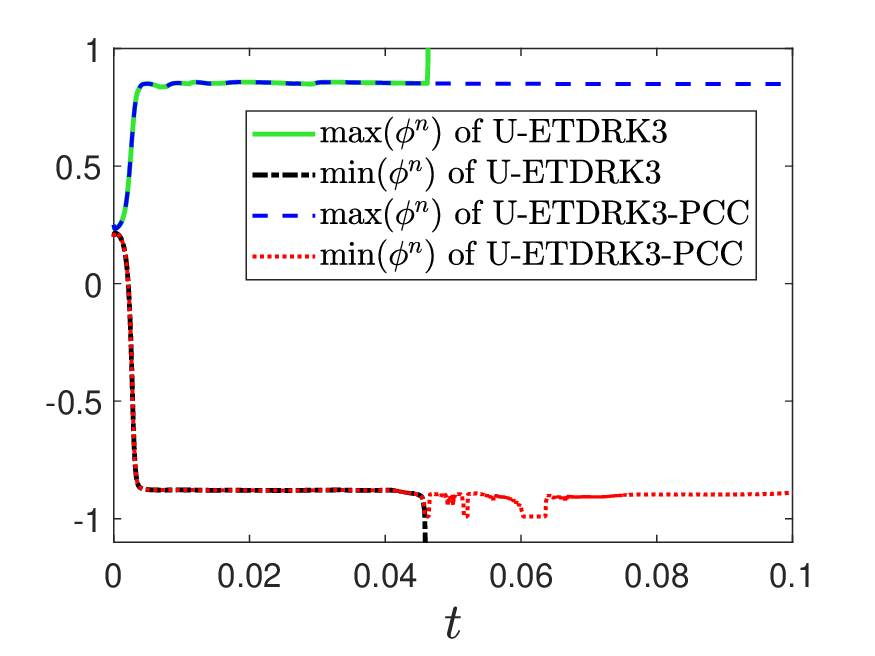}
        \caption{(Example \ref{Ex 3}) Evolution of $\mathcal{E}[\phi^{n}]$ (1st row) and extreme values of numerical solutions (2nd row) in different schemes. Left column: $S=1/{\varepsilon^2}$ and $\tau=10^{-4}$; Middle column: $S=1.2/{\varepsilon^2}$ and $\tau=10^{-3}$; Right column: $S=1.75/\varepsilon^2$ and $\tau=10^{-4}$.}
        \label{fig:ex3_Boundary_stability}
    \end{figure}
    
    \begin{figure}[h!] 
        \centering
        \includegraphics[width=0.24\textwidth, height=0.2\textwidth]{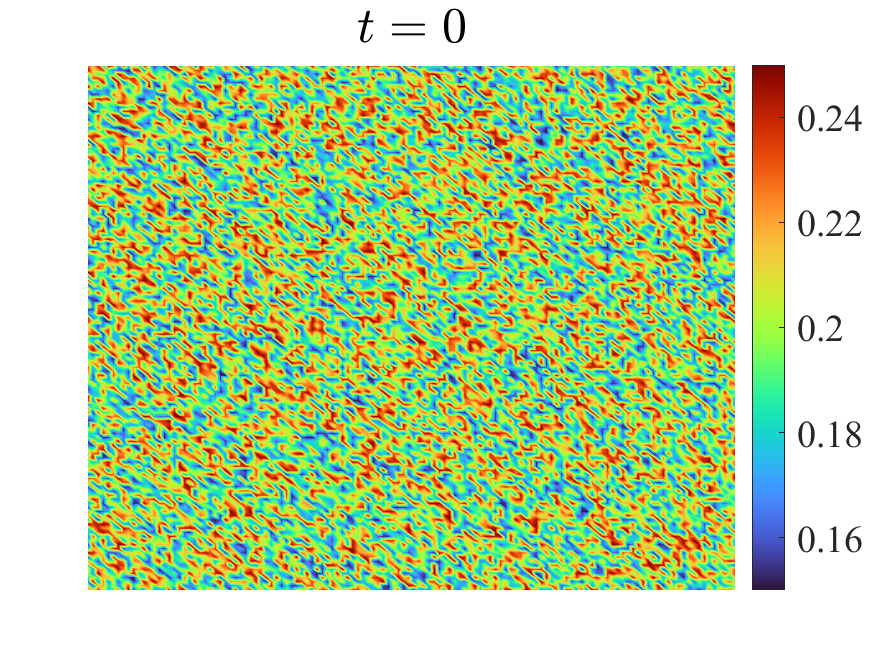} 
        \includegraphics[width=0.24\textwidth, height=0.2\textwidth]{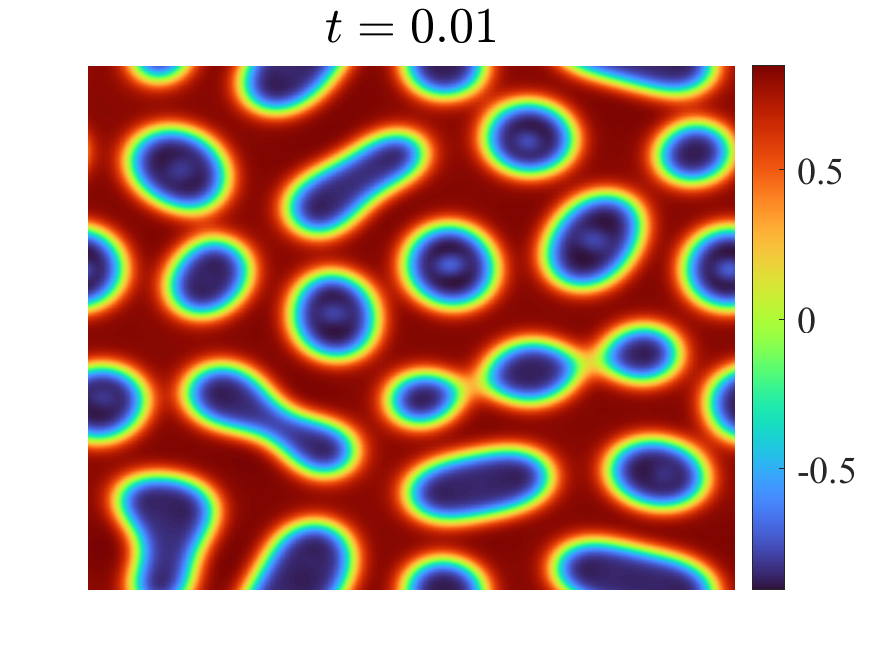} 
        \includegraphics[width=0.24\textwidth, height=0.2\textwidth]{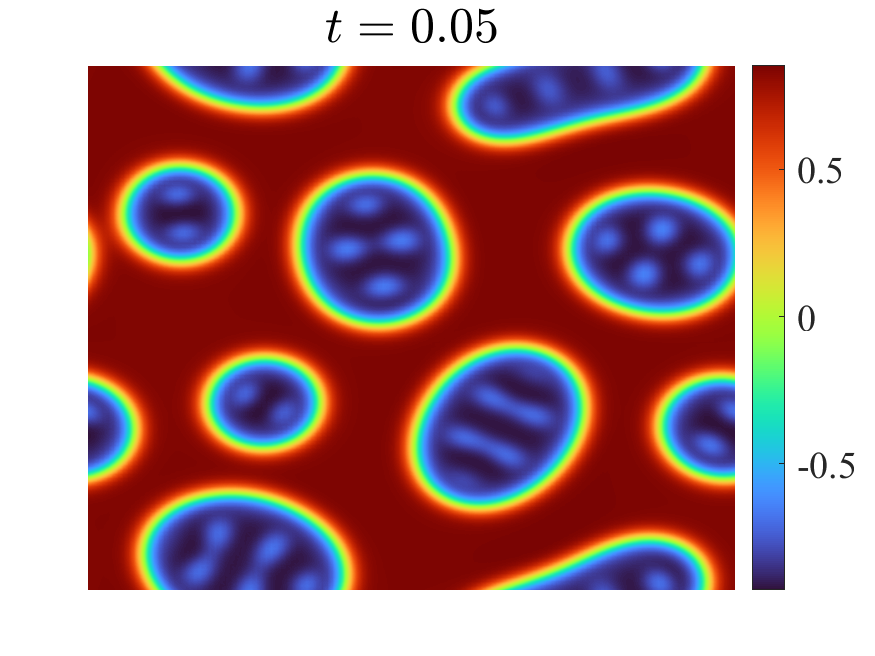} 
        \includegraphics[width=0.24\textwidth, height=0.2\textwidth]{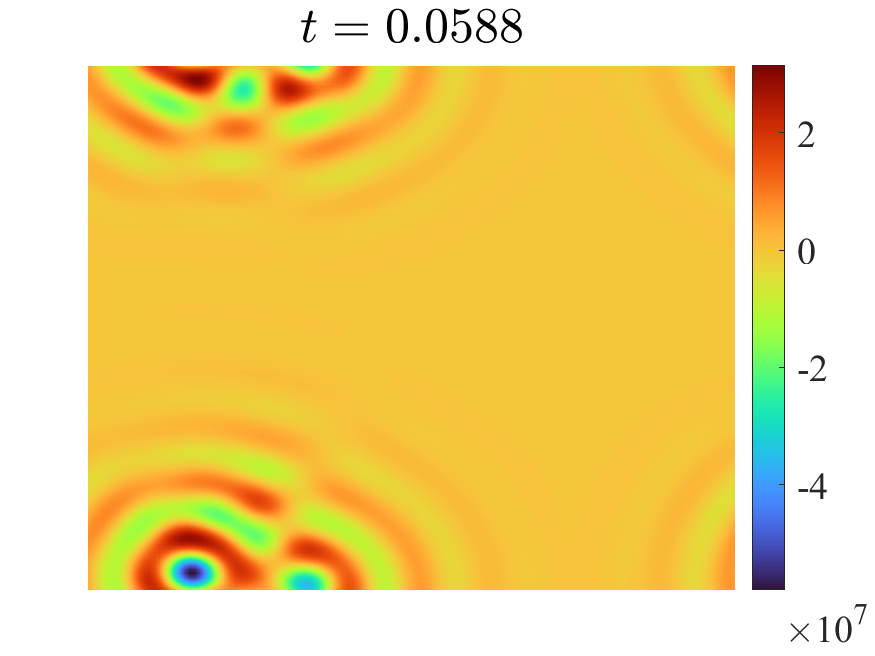}
        \\
        \includegraphics[width=0.24\textwidth, height=0.2\textwidth]{Figures/ex3_snapshot_0_ETDRK2_PC.eps} 
        \includegraphics[width=0.24\textwidth, height=0.2\textwidth]{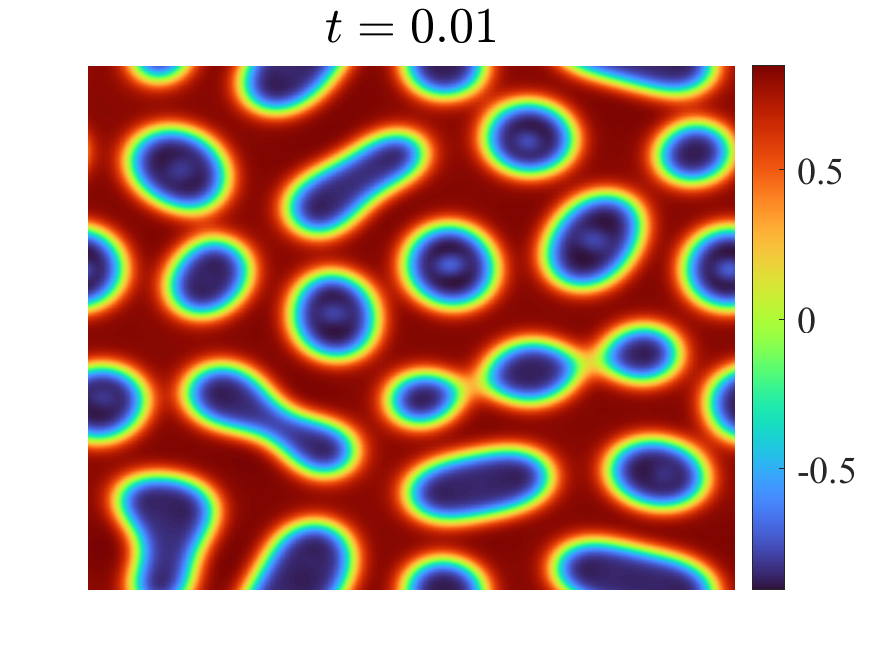} 
        \includegraphics[width=0.24\textwidth, height=0.2\textwidth]{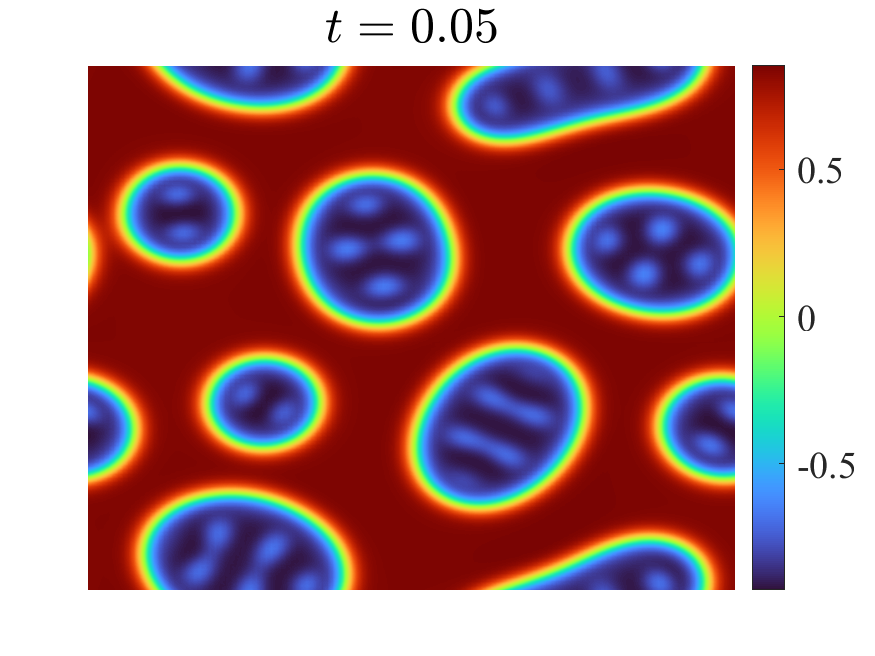} 
        \includegraphics[width=0.24\textwidth, height=0.2\textwidth]{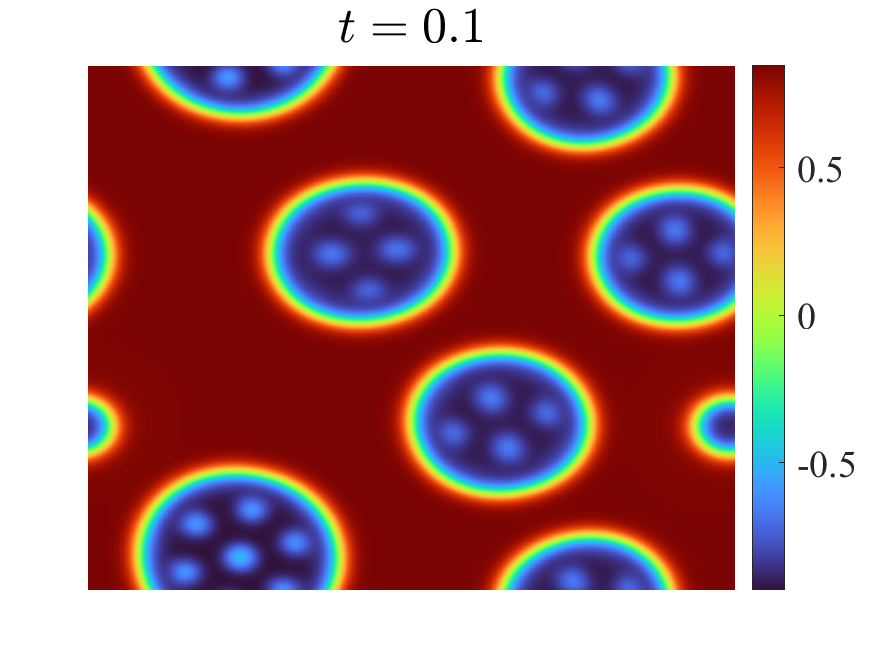}
        \caption{(Example \ref{Ex 3}) Contour plots of the numerical solution $\phi^n$ at different time with $S=1/\varepsilon^2$ and $\tau=10^{-4}$. 1st row: ETDRK2; 2nd row: ETDRK2-PCC.}
        \label{fig:ex3_snapshot s=2}
    \end{figure}
    
    Numerically, we solve this example by the proposed PCC with the bound `$\beta$' in \eqref{fml. compute phi star} for Corrector 3 set as $\pm(\beta-\delta)$ and  $\delta=0.01$.
    Theoretically, for the Flory-Huggins potential to satisfy the condition in Assumption~\ref{ass. f Lipschitz}, the value of $C_{L}$ tends to be quite large, which imposes stringent requirements on $S$. 
    When a relatively small value is selected for $S$, the original ETDRK scheme could exhibit a noticeable energy instability and a numerical blow-up may occur. In contrast, our proposed PCC scheme is capable of maintaining energy stability while preserving the MBP. These are well illustrated by the numerical results presented in Figs.~\ref{fig:ex3_Boundary_stability}--\ref{fig:ex3_snapshot s=2}. 
	
    \begin{example} \label{Ex 4}
       Consider a three-dimensional example with the initial data
        \begin{align*}
            \phi(x,y,0) =& \frac{1}{2}\phi_0(x,y),\mbox{ with $\phi_0$ defined in \eqref{phi0 def}}.
        \end{align*}
        The bound is set as $\pm(\beta-\delta)$ and  $\delta=0.01$. We consider the computational domain $\Omega=[-\pi,\pi]^{3}$ with the $128\times 128$ uniform spatial grids and choose  $\varepsilon = 0.1$. The stabilizing parameter of ETDRK scheme is chosen as $S = 5/\varepsilon^2$.
    \end{example}
	
    \begin{figure}[h!] 
        \centering
        \includegraphics[width=0.32\textwidth]{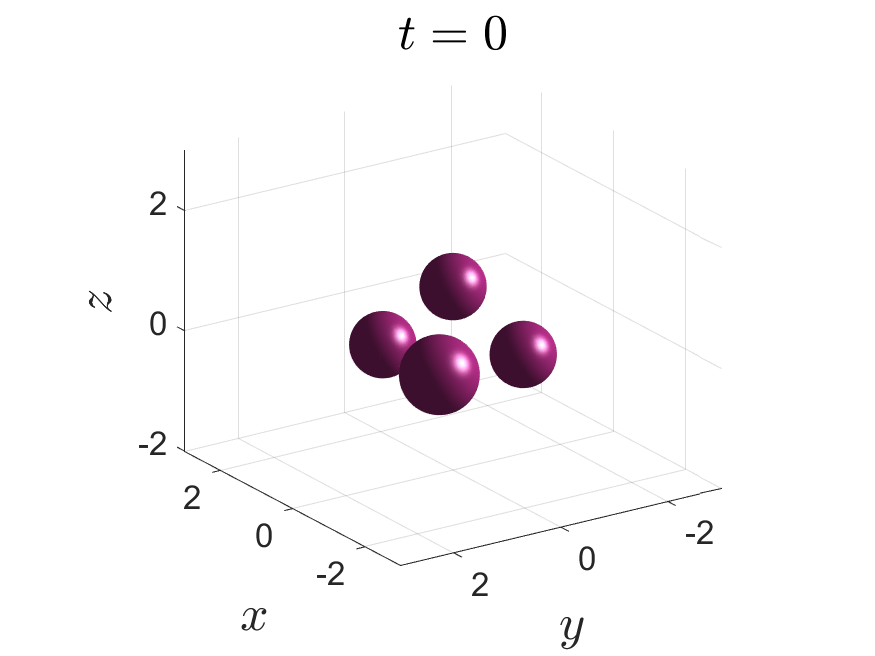} 
        \includegraphics[width=0.32\textwidth]{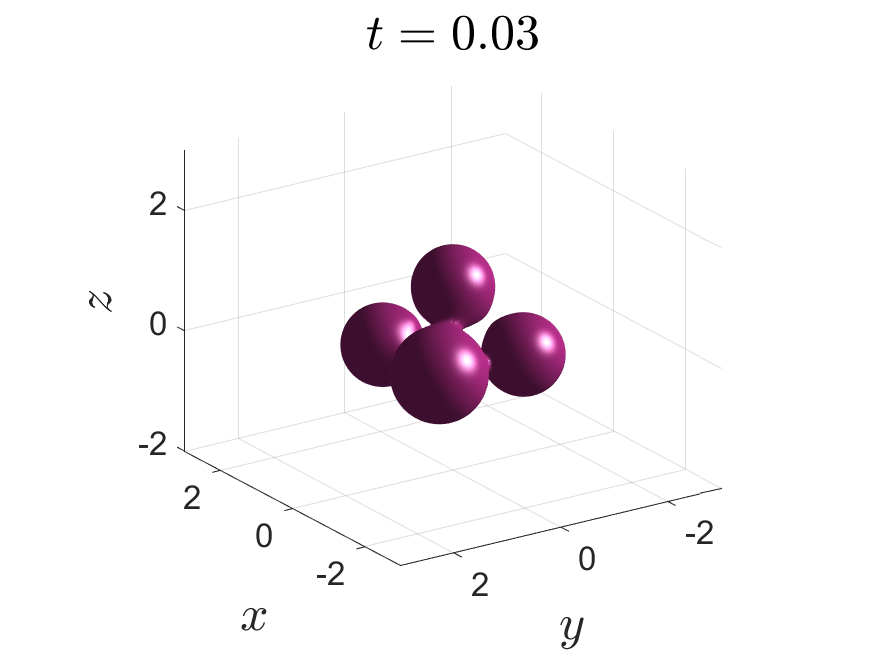} 
        \includegraphics[width=0.32\textwidth]{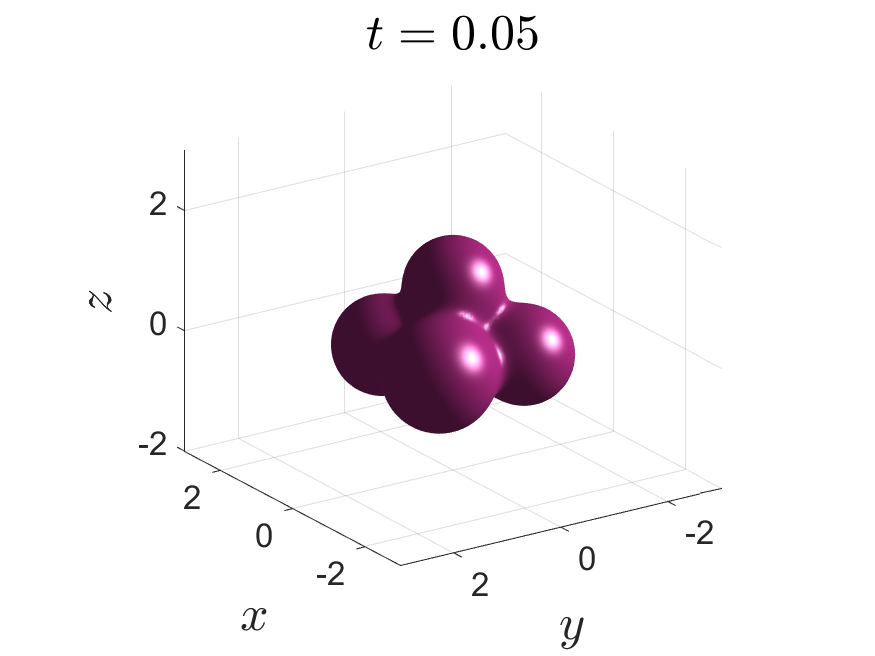} \\
        \includegraphics[width=0.32\textwidth]{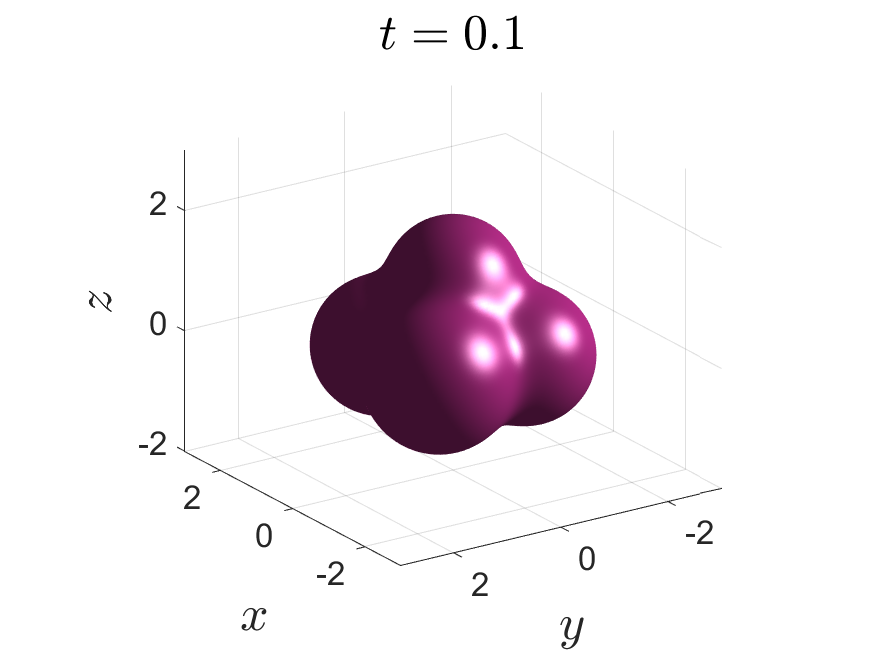}
        \includegraphics[width=0.32\textwidth]{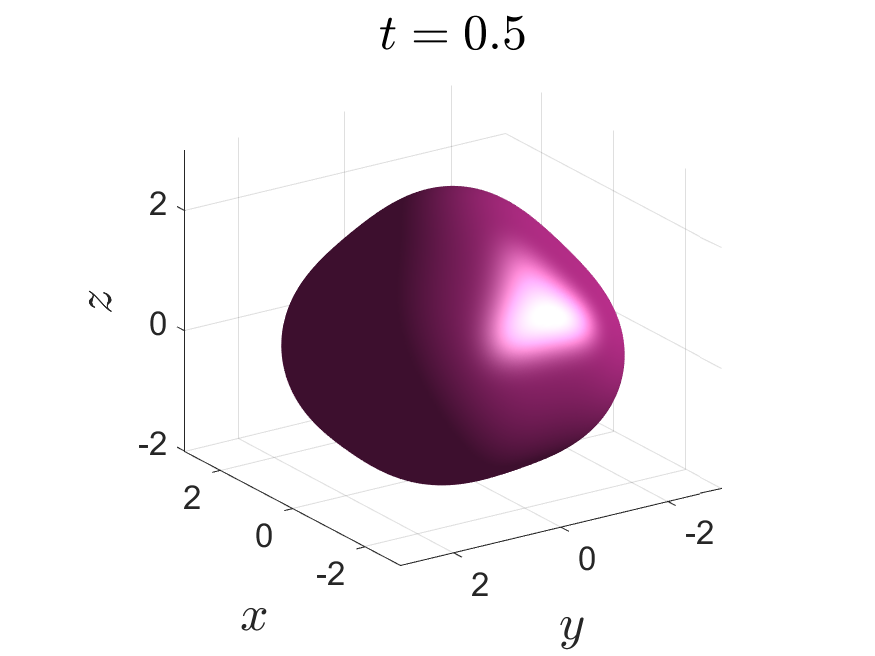}
        \caption{(Example \ref{Ex 4}) Snapshots of the numerical solution from U-ETDRK3-PCC: the iso-surfaces plots correspond to $\phi^{n} = 0$ at time  $t = 0$, $0.03$, $0.05$, $0.1$, and $0.5$ under $\tau = 0.01$.}
        \label{fig:ex4_snapshot s=3}
    \end{figure}
    
    \begin{figure}[h!] 
        \centering
        \includegraphics[width=0.4\textwidth]{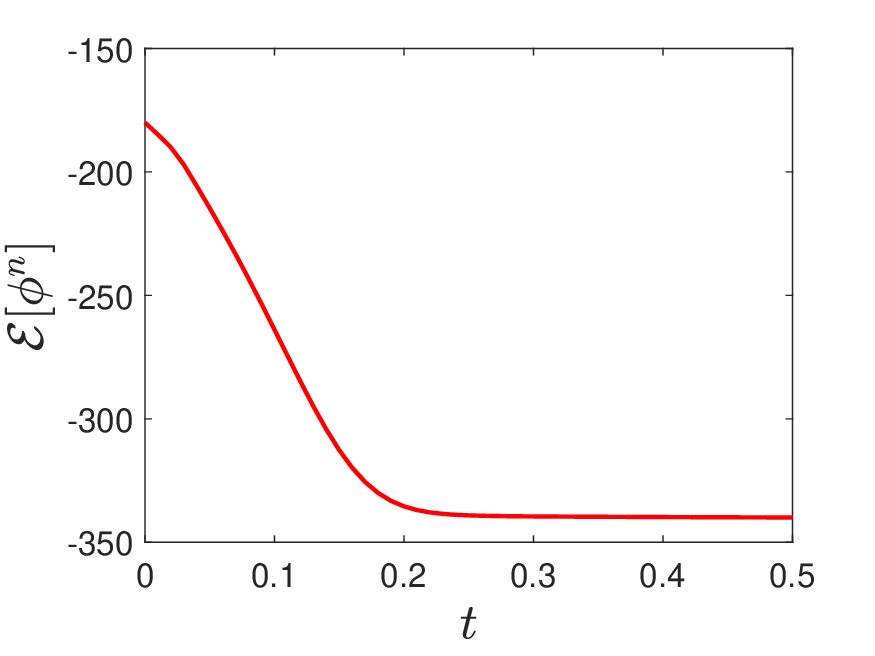}
        \includegraphics[width=0.4\textwidth]{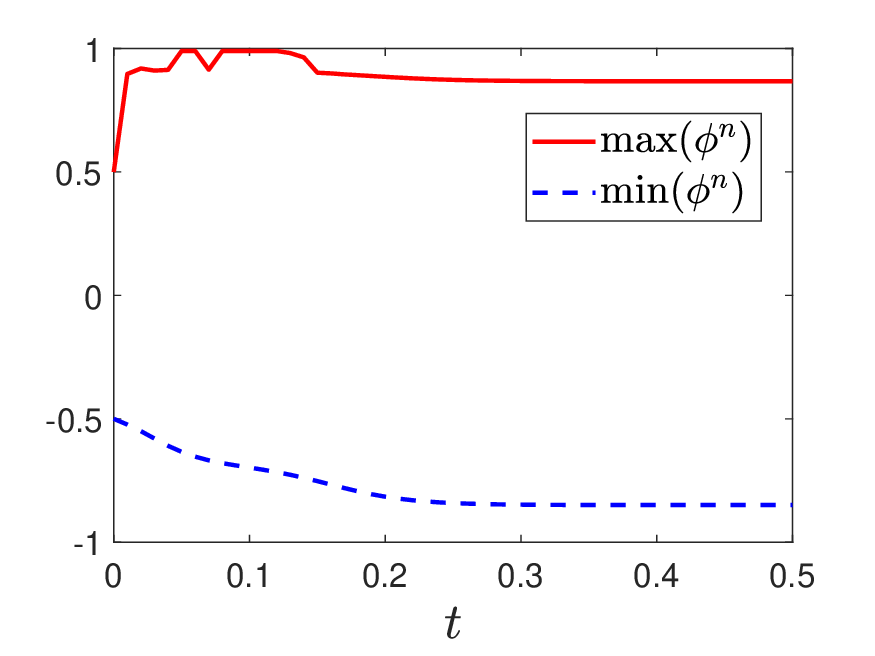}
        \caption{(Example \ref{Ex 4}) Time evolutions of energy, upper and lower bounds of the numerical solution $\phi^{n}$ from U-ETDRK3-PCC scheme with  $\tau = 0.01$.}
        \label{fig:ex4_stability}
    \end{figure}
    
    Fig.~\ref{fig:ex4_snapshot s=3} shows the dynamics obtained via U-ETDRK-PCC for Example~\ref{Ex 4}, and Fig.~\ref{fig:ex4_stability} demonstrates the desired bound and energy behavior. 
    
    \section{Conclusion} \label{sec. 6}
    For gradient flows,  the existing structure-preserving schemes are difficult to achieve arbitrary high-order accuracy in time while preserving maximum-principle (MBP) and energy dissipating simultaneously. In this work, we developed a novel framework for constructing arbitrary high-order structure-preserving schemes in time  which simultaneously preserving energy stability and bound-preserving.  The key idea is that we introduce a new KKT-condition to enforce the constraint of energy  inequality and obtain a new equivalent system to construct structure-preserving schemes for gradient flows. Then we establish a new framework for constructing energy dissipating and MBP-preserving schemes which can be combined with most existing numerical methods, for example, ETD Runge Kutta schemes, BDF schemes. The unique solvability analysis of our numerical schemes are given by rewriting the correction steps as optimization problems. Finally, we give a robust error analysis for the ETD Runge Kutta schemes based on our new framework. Enough numerical comparisons with the existing popular schemes are shown that our structure-preserving schemes can avoid numerical oscillations and capture the accurate evolution of energy for gradient flows.

\section{Acknowledgments}
The authors would like to thank Professor Jie Shen, Chun Liu and Weizhu Bao for their valuable suggestions and assistance.

    \bibliographystyle{siamplain}
    {\bibliography{ref}}
\end{document}